\documentclass[a4paper,12pt]{article}
        
\usepackage{macro_Sam} 
\author{Samuel TAPIE}
\date{06/11/2009}
\title{Graphes, moyennabilité et bas du spectre de variétés topologiquement infinies}

\begin{document}
\maketitle

\selectlanguage{francais}

\begin{abstract}
A partir d'un graphe $G$ à valence constante $v$ et d'une variété $C$ (non compacte) à $v$ composantes de bords (la <<cellule>>), on construit une variété $M$ en disposant des copies de $C$ sur chaque sommet de $G$ et en identifiant les bords qui correspondent à des arêtes de $G$. Nous obtenons ainsi une classe de variétés topologiquement infinies qui généralise les variétés périodiques sur laquelle étudions le bas du spectre du Laplacien. Notre principal résultat est que, lorsque la première fonction propre de $C$ se prolonge à $M$, le bas du spectre de $M$ est égal à celui de $C$ si et seulement si le graphe $G$ est moyennable. Lorsque $G$ n'est pas moyennable, nous contrôlons explicitement l'écart entre ces deux bas du spectre à l'aide de la combinatoire de $G$. Cela nous permet en particulier de montrer que si $p : M\ra N$ est un revêtement riemannien, et la métrique de $N$ est générique, alors $\lambda_0(M)\geq \lambda_0(N)$ avec égalité si et seulement si le groupe de revêtement est moyennable. Ceci généralise notamment un résultat de R. Brooks.
\end{abstract}

\tableofcontents{}

\section*{Introduction}
L'étude du lien entre les propriétés géométriques d'une variété riemannienne et ses propriétés spectrales est un domaine actif depuis plusieurs décénies, à l'origine de nombreux outils analytiques appliqués à la géométrie. Nous nous intéressons ici à un cas peu étudié, celui de variétés topologiquement infinies, c'est-à-dire dont le groupe fondamental n'est pas de type fini. Les revêtements riemanniens fournissent un grand nombre d'exemples ; nous en présentons une classe un peu plus vaste : les variétés construites à partir d'une infinité de copies d'une unique cellule de base, recollées en suivant un graphe. 

\pgh

A la première section, nous présentons les outils dont nous aurons besoin : une introduction à la théorie spectrale, quelques résultats sur les domaines fondamentaux de revêtements riemanniens, des méthodes de contrôle de la géométrie au voisinage d'une hypersurface, et quelques notions de combinatoires des graphes. A la deuxième section, nous présentons la construction des variétés \emph{$G$-périodiques} (voir Définition \ref{def:GPeriod}), qui sont construites à partir d'une \emph{cellule} à bords, dont on place une copie à chaque sommet d'un graphe $G$ à valence constante en identifiant les composantes de bord qui correspondent à une arête du graphe. Nous détaillons dans cette section quelques exemples, en particulier certaines surfaces hyperboliques de genre infini et les revêtements riemanniens, puis nous y présentons nos résultats. Sur la cellule, nous considérons le Laplacien avec condition de Neumann au bord. Notre théorème principal est le suivant :

\begin{theo*}
	Soit $C$ une cellule de valence $v$ qui admet une première fonction propre $\phi_0$ et un trou spectral $\eta = \lambda_1(C)-\lambda_0(C)>0$. Pour graphe $G$ de valence constante $v$ et toute variété $G$-périodique $M$ de cellule $C$ où $\phi_0$ se recolle bien, si $G$ est moyennable alors $\lambda_0(M) = \lambda_0(C)$. Si $G$ n'est pas moyennable, il existe une constante $A$ ne dépendant que de la géométrie sur les bords de $C$ telle que pour tout graphe $G$ de même valence $v$, et 
	$$\lambda_0(M)-\lambda_0(C)\geq A\eta\mu_0(G)>0,$$
où $\mu_0(G)$ est le bas du spectre combinatoire de $G$. De plus, lorsque chaque cellule n'admet qu'un nombre fini de voisins, on a
$$\lambda_0(M)-\lambda_0(C)\leq A'\g h(G),$$
où $\g h(G)$ est la constante de Cheeger de $G$ et $A'$ ne dépend que de la géométrie sur le bord de $C$ et de $\lambda_0(C)$.
	\end{theo*}

Si l'on ne cherche pas à contrôler l'écart entre les bas du spectre, ce résultat s'énonce simplement :

\begin{coro*}
Soit $C$ une cellule de valence $v$ qui admet une première fonction propre $\phi_0$ et un trou spectral strictement positif. Pour toute variété $G$-périodique $M$ de celulle $C$ où $\phi_0$ se recolle bien, $\lambda_0(M) \geq \lambda_0(C)$ avec égalité si et seulement si le graphe $G$ est moyennable.
\end{coro*}

Notre résultat s'applique en particulier aux revêtements qui admettent un domaine fondamental adapté :
	\begin{coro*}
Soit $N$ une variété dont le rayon d'injectivité et le trou spectral sont strictement positifs. Notons $\phi_0$ sa première fonction propre. Soit $p : M\ra N$ un revêtement riemannien. Supposons qu'il existe un domaine fondamental pour l'action du groupe de revêtement $\Gamma$ de $p$ à bord $\Cl C^1$ par morceaux sur lequel le relevé de $\phi_0$ vérifie les conditions de Neumann. Alors $\lambda_0(M) \geq \lambda_0(N)$, avec égalité si et seulement si le groupe $\Gamma$ est moyennable.
	\end{coro*}

Nous avons montré dans \cite{Tap09b} que si la courbure sectionnelle et les $k\geq \frac{n}{2}$ dérivées de la courbure de $N$ sont bornées et la métrique de $N$ est générique, alors la première fonction propre est de Morse et pour tout revêtement, il existe un domaine fondamental pour l'action du groupe de revêtement à bord $\Cl C^1$ par morceaux sur lequel le relevé de $\phi_0$ vérifie les conditions de Neumann. Nous obtenons donc le corollaire :
\begin{coro*}\label{coro*:SpecRevet}
Soit $N$ une variété non compacte de dimension $n$ dont le trou spectral et le rayon d'injectivité sont strictement positifs. Si la courbure sectionnelle et les $k\geq \frac{n}{2}$ premières dérivées du tenseur de courbure de $N$ sont bornées et si la métrique est générique, alors pour tout revêtement riemannien $p : M\ra N$ de groupe d'automorphisme $\Gamma_p$, on a
$$\lambda_0(M) \geq \lambda_0(N)$$
avec égalité si et seulement si $\Gamma_p$ est moyennable.
\end{coro*}

Nous détaillons ces résultats au Paragraphe \ref{ssec:Result}, après en avoir défini tous les termes. A la troisième section, nous les démontrons. Enfin, nous présentons quelques applications de nos techniques, en particulier le cas des revêtements et une méthode de calcul approché du bas du spectre à partir des propriétés combinatoires d'un découpage de $M$. 
	
\pgh

Cette étude est une partie de ma thèse, effectuée sous la direction de Gérard Besson et Gilles Courtois. Il a été possible avant tout grâce à leur patience et leurs encouragements répétés. J'ai été invité plusieurs fois à exposer ces résultats : un grand merci en particulier à Constantin Vernicos, Françoise Dalb'o et Bruno Colbois pour ces séjours fructueux qui ont permis la maturation de ce travail. Merci également au rapporteur pour ses remarques lors d'une première soumission de cet article, qui ont abouti à un net approfondissement de ses résultats, et à Gilles Carron qui a patiemment relevé les erreurs de calculs qui y restaient lors de la relecture de mon manuscrit de thèse. 

\section{Préliminaires}

Tout au long de cet article, nous considèrerons des \ind{variétés riemanniennes} de dimension finie $n\geq 2$, dénombrables à l'infini et nous supposerons toujours (sauf mention explicite) qu'elles sont \textbf{connexes} et \textbf{complètes}. Lorsque nous considèrerons un \ind{revêtement} $p : M\ra N$, nous supposerons toujours que la structure riemannienne de $M$ est relevée de celle de $N$. Un <<revêtement>> signifiera toujours un \ind{revêtement galoisien}. 
	        
        \subsection{Eléments de Théorie spectrale}\label{ssec:PreSpec}
        
Nous présentons, dans ce paragraphe, les éléments de théorie spectrale nécessaires à la compréhension de nos travaux. Le lecteur intéressé par un exposé plus complet pourra se référer par exemple à \cite{Cha84}.
    
Nous appellerons simplement \ind{Laplacien} l'opérateur de \ind{Laplace-Beltrami} $\Delta_g$ défini sur toute fonction $\Cl C^2$ sur $M$ (à valeur réelle) par
    $$\Delta_g f = \gdiv(\grad f) := -\Trace((\nabla_g)_. (\nabla_g f)),$$
où $\nabla_g$ est la connexion de Levi-Civita associée à la métrique $g$. Lorsqu'il n'y aura pas d'ambigüité possibles sur la métrique utilisée, nous omettrons l'indice $g$. Nous dirons qu'une fonction $\phi$ de classe $\Cl C^1$ sur $M$ vérifie les \indb{conditions}{de Dirichlet} si elle est nulle sur $\bd M$, et les \indb{conditions}{de Neumann} si en tout point de $\bd M$ son gradient est tangent au bord. Lorsque l'espace tangent au bord n'est défini qu'en dehors d'un ensemble de mesure nulle $E$ (par exemple en dehors des \emph{coins} de $\bd M$, voir paragraphe \ref{sec:CoinsRevet}), nous dirons encore qu'une fonction vérifie les conditions de Neumann si elle les vérifie en dehors de $E$. Nous imposerons par la suite des conditions de Neumann dès que nous serons sur des variétés à bord.
    
\pgh

On note $L^2(M)$ l'ensemble des fonctions sur $M$ à valeurs réelles de carré sommable, $\Cl H^1(M)$ l'ensemble des fonctions de $L^2$ dont le gradient au sens des distributions est un champ de vecteur de carré sommable, et $\Cl H^2(M)$ l'ensemble des fonctions de $\Cl H^1(M)$ dont le Laplacien (au sens des distributions) est une fonction $L^2$. Le Laplacien avec conditions de Neumann s'étend alors en un opérateur non borné sur $L^2(M)$ (voir par exemple \cite{ReeSim80} section VIII), dont le domaine de définition maximal est $\Cl H^2(M)$. 
    
Nous noterons par la suite $\norm{.}$ la norme $L^2$, en précisant le domaine d'intégration en cas d'ambigüité. D'après la formule de Green, le Laplacien avec conditions de Neumann est toujours l'opérateur associé à la forme quadratique $\norm{\nabla f}^2$ sur son domaine $\Cl H^2$, et est donc positif ou nul. On appelle \ind{spectre du Laplacien} de $M$ l'ensemble des $\lambda\in\Bb R$ tels que l'opérateur
$$\Delta-\lambda : \Cl H^2\ra L^2$$
n'est pas inversible ; il est inclus dans $\Bb R_+$ puisque $\Delta$ est positif. Le \ind{bas du spectre} est sa borne inférieure. Un résultat classique d'analyse hilbertienne nous donne la caractérisation suivante du bas du spectre :
    
\begin{prop}[Principe du Min-Max]\label{prop:MinMax}
Le bas du spectre (avec condition de Neumann si $M$ a un bord) est donné par
$$\lambda_0(M) = \inf_f\left\{\frac{\norm{\nabla f}^2}{\norm{f}^2}\right\}$$
    o\`u $f$ parcourt l'ensemble des fonctions de $\Cl H^1(M)$.
\end{prop} 

    Pour toute fonction $f$ de $\Cl H^1$ (par exemple continue et $\Cl C^1$ par morceaux), on appelle $\frac{\norm{\nabla f}^2}{\norm{f}^2}$ son \ind{quotient de Rayleigh}. Nous dirons qu'une fonction $f$ est \ind{$\lambda$-harmonique} si $\Delta f = \lambda f$, et \indb{fonction}{propre} du Laplacien (avec condition de Neumann) associée à la valeur propre $\lambda$ si elle est dans $\Cl H^2$ et $\lambda$-harmonique. Une valeur propre est un point du spectre, donc nécessairement positive ou nulle. Si $M$ est \emph{compacte}, on montre que le spectre est l'ensemble (discret) de la suite de ses valeurs propres, qui sont alors de multiplicité finie. Pour $M$ de \emph{volume fini}, $\lambda_0 = 0$ est valeur propre associée aux fonctions constantes. Lorsque $M$ n'est pas de volume fini, l'existence de fonctions propres (et donc de valeurs propres) n'est pas assurée. Le résultat suivant, que nous utiliserons souvent par la suite, regroupe plusieurs théorèmes classiques :

    \begin{theo}\label{theo:EigFunc}
    S'il existe une fonction $\phi_0\in \Cl H^1(M)$ telle que son quotient de Rayleigh soit égal à $\lambda_0(M)$, alors $\phi_0$ est fonction propre du Laplacien, avec condition de Neumann si $\bd M\neq \vd$. Elle est de signe strictement constant sur $M$, de classe $\Cl C^\infty$ sur $\inter{M}$. Toute fonction propre du Laplacien associée à la valeur $\lambda_0$ est alors proportionnelle à $\phi_0$. 
    \end{theo}

L'unicité découle du fait que le signe d'une fonction propre associée au bas du spectre est constant. Ceci s'obtient en montrant que le quotient de Rayleigh d'une fonction est plus grand que celui de sa valeur absolue, et par le Principe du Maximum une fonction propre positive ne peut s'annuler en un point intérieur de la variété. Si $M$ a un bord, une fonction qui réalise le minimum des quotients de Rayleigh vérifie automatiquement les conditions de Neumann (et est donc fonction propre pour le Laplacien avec condition de Neumann), et sa stricte positivité sur $\bd M$ découle alors du Principe du Maximum Fort (voir \cite{ProWei84}, Chapitre 2).

    Lorsque nous parlerons de \emph{la première fonction propre}, nous sous-entendrons donc par là que nous parlons de l'unique fonction propre positive normalisée associée au bas du spectre (quand elle existe).

    Nous aurons besoin des propriétés élémentaires du bas du spectre suivantes, dont nous laissons la démonstration au lecteur à partir du Min-Max et du théorème précédent :

    \begin{prop}\label{prop:MonoNeum}
    Soit $M = \bigcup_i M_i$ une partition de la variété $M$ en morceaux d'intérieur non vide dont les bords sont $\Cl C¹$ par morceaux. On a $$\lambda_0(M)\geq\inf_i\{\lambda_0(M_i)\}.$$
    \end{prop}
    
    \begin{prop}\label{prop:InvEigFunc}
    Soit $M$ une variété riemannienne (éventuellement à bord) invariante par un
    groupe d'isométries $G$. Si le bas du spectre de $M$ est atteint par une fonction $\phi_0\in L^2(M)$, alors $\phi_0$ est invariante par $G$.
    \end{prop}

    \begin{defi}\label{def:Trou}
    Soit $M$ une variété riemannienne, nous dirons que $M$ a un \emph{trou spectral strictement positif} si les deux conditions suivantes sont vérifiées :
    \bi
    \item le bas du spectre $\lambda_0$ de son Laplacien est associé à une fonction propre $\phi_0\in \Cl H^1(M)$ ;\\
     on note alors
    $$\lambda_1 = \inf\left\{\frac{\norm{\nabla h}^2}{\norm{h}^2} ; \int_M\phi_0 h = 0\right\}$$
    o\`u $h$ parcourt l'ensemble des fonctions de $\Cl H^1(M)$ orthogonales à $\phi_0$ ;
    \item on a $\eta = \lambda_1-\lambda_0> 0$.
    \ei
    On appelle alors $\eta$ le \ind{trou spectral} de $M$.
    \end{defi}
    
    Lorsque $M$ est compacte (avec ou sans bord) le trou spectral de $M$ est donc strictement positif. On peut également exprimer cette définition à l'aide du \emph{spectre essentiel} :
 
 \begin{defi}\label{def:SpecEss}
 Soit $(M,g)$ une variété riemannienne et $\Delta$ son Laplacien. On dit que $\lambda\in\Bb R$ est dans le \indb{spectre}{essentiel} de $\Delta$ s'il existe une suite de fonctions $\psi_n\in\tilde{\Cl H^2}(M)$ toutes orthonormales pour le produit scalaire de $\tilde{\Cl H^2}$ telles que 
 $$\lim_{n\ra\infty}\norm{\Delta\psi_n - \lambda\psi_n}_{L^2(M)} = 0.$$
 On note $\lambda_0^{ess}(g)$ l'infimum du spectre essentiel de $\Delta$.
 \end{defi}

\begin{prop}
Les éléments du spectre qui ne sont pas dans le spectre essentiel sont des valeurs propres de multiplicité finie. En particulier, le trou spectral est strictement positif si et seulement si $$\lambda_0(g)<\lambda_0^{ess}(g).$$
\end{prop}

La première affirmation est un résultat classique sur le spectre essentiel (voir \cite{Don81}, §2). La seconde découle du fait que le bas du spectre ne peut être une valeur propre que de multiplicité 1 d'après le Théorème \ref{theo:EigFunc}.
\begin{ex}
Avec notre définition, l'espace hyperbolique $\Bb H^n$ a un trou spectral nul : le bas du spectre est aussi son bas du spectre essentiel, égal à $\frac{(n-1)^2}{4}$. Il est démontré dans \cite{LaxPhi82} qu'une variété hyperbolique de dimension $n$ géométriquement finie non compacte a un trou spectral strictement positif si et seulement si $\lambda_0<\frac{(n-1)^2}{4}.$ C'est le cas par exemple des surfaces de Riemann dont le bord du coeur convexe a été suffisamment pincé, ou des variétés hyperboliques de dimension 3 géométriquement finie et acylindriques (voir \cite{CanMinTay99}).
\end{ex}
Par la suite, nous étudierons majoritairement des variétés non compactes, dont le trou spectral sera souvent supposé strictement positif. Nous verrons que c'est alors une hypothèse cruciale de nos démonstrations.

\subsection{Variétés à coins et revêtements}\label{sec:CoinsRevet}

Soit $M$ une variété de dimension $n$ à bord, nous dirons que $\bd M$ est \emph{$\Cl C^1$ par morceaux}, ou que $M$ est une \emph{variété à coins}, s'il existe un atlas $\Cl C^1$ de $M$ dans $\Bb R^n$ tel que l'image de tout ouvert de $M$ est un ouvert du cadran
$$\{(x_1,...,x_n,) ; x_1\geq 0,...,x_n\geq 0\}.$$ Le bord $\bd M$ est alors une réunion localement finie de variétés à coins de dimension $n-1$. Les points singuliers de $\bd M$ forment un ensemble de mesure nulle dans $\bd M$. On montre alors que les  formules d'intégrations de Stokes et Green, et toute la théorie spectrale vue à la section \ref{ssec:PreSpec}, restent valable dans le cas des variétés à coins. 
\pgh

Nous allons chercher par la suite à étudier le bas du spectre de revêtements riemanniens $p : M\ra N$. Nos démonstrations utiliserons un \emph{domaine fondamental} adapté à une fonction sur $N$ : nous présentons maintenant un mode de construction de tels domaines.

Soit $p : M\ra N$ un revêtement riemannien, et $\Gamma = \pi_1(N)/\pi_1(M)$ son groupe d'automorphisme. Rappelons qu'un ensemble fermé $D\subset M$ est un \ind{domaine fondamental} pour le revêtement $p$ si $p(D) = N$, donc $\Gamma(D) = M$, et pour tout $g\in G\bs \{id\}$, $p(\inter{D})\cap \inter{D} = \vd$. On a alors la propriété classique suivante :

\begin{prop}\label{prop:GeneGroupeRev}
Soit $p : M\ra N$ un revêtement galoisien dont le groupe d'automorphisme est de type fini. On suppose qu'il existe un domaine fondamental $D\subset M$ dont le bord est $\Cl C^1$ par morceaux. Alors le groupe $G$ est engendré par un nombre fini d'éléments $S_D\subset G$ tels que pour tout $g\in S_D$, $\bar D\cap g(\bar D)$ est de codimension 1. On peut de plus supposer $S_D$ symétrique.
\end{prop}

Nos domaines fondamentaux seront désormais toujours supposés d'intérieur connexe. Pour tout revêtement d'une variété lisse, il existe un domaine fondamental pour l'action du groupe d'automorphisme, à bord $\Cl C^1$ par morceaux. Nous construisons un tel domaine en utilisant un outil classique de topologie différentielle : les fonctions de Morse.

\begin{defi}\label{def:Morse}
Soit $f : M\ra \Bb R$ une fonction de classe $\Cl C^2$, $f$ est une \indb{fonction}{de Morse} si et seulement si en tout point critique de $f$, le hessien $d^2f$ est de rang maximal.
\end{defi}


Les points critiques d'une fonction de Morse sont isolés, et plus généralement, son gradient est simple à étudier. Nous utiliserons ces fonctions par l'intermédiaire du résultat suivant :

\begin{theo}\label{theo:DomFondMorse}
Soit $(N,g)$ une variété riemannienne de dimension $n$, et $p : M\ra N$ un revêtement riemannien galoisien de groupe d'automorphisme $\Gamma$. S'il existe une fonction de Morse $f : N\ra \Bb R_+^*$, telle que pour tout $a>0$, 
$$N_a = \left\{x\in N ; f(x)\geq a\right\}$$
soit compact, alors il existe un domaine fondamental $C\subset M$ pour l'action de $\Gamma$ à bord $\Cl C^1$ par morceaux et d'intérieur connexe tel que le gradient du relevé de $f$ à $M$ soit tangent à $\bd C$. De plus, si $f$ admet un nombre fini de point critiques, pour tout $1\leq k\leq n,$ le bord de $C$ a un nombre fini de composantes lisses de dimension $k$.
\end{theo}

Ce théorème est un cas particulier d'un résultat plus général, le Théorème \ref{theo:DomFondStrat} que nous démontrons ci-dessous. Le domaine fondamental ainsi construit est donc stable sous le flot de gradient de $f$. De façon équivalente, la fonction $f$ vérifie les conditions de Neumann sur $\bd C$. Il existe toujours une fonction vérifiant les hypothèses du théorème précédent, par un résultat classique en Théorie de Morse :

\begin{theo}
Sur toute variété différentiable $N$ il existe une fonction de Morse $f : N\ra \Bb R^*_+$ telle que pour tout $a>0$, $N_a = \left\{x\in N ; f(x)\geq a\right\}$ soit compact.
\end{theo}

Il s'agit du Corollaire 6.7 de \cite{Mil63}, dont la démonstration est détaillée tout au long du Chapitre 6 du livre de J. Milnor. Le Théorème \ref{theo:DomFondMorse} associe donc à toute fonction de Morse un domaine fondamental qui lui est adpaté. Pour associer un tel domaine à une fonction $f$, la propriété de Morse n'est pas nécessaire : il suffit qu'il existe une partition localement finie de la variété de départ en morceaux simplement connexes, stable par le flot de gradient de $f$.

\begin{defi}\label{def:GradStrat}
Soit $(N,g)$ une variété riemannienne et $f : N\ra \Bb R$ une fonction de classe $\Cl C^1$. Nous dirons que $f$ a un \ind{gradient simplement stratifié} si $N$ se décompose en $N = \coprod_i N_i,$
où les $N_i$ sont des \emph{sous-variétés ouvertes disjointes} de classe $\Cl C^1$, \emph{simplement connexes} et \emph{stables par le flot de gradient} de $f$ associé à $g$, et où la partition $(N_i)_i$ est \emph{localement finie}.
\end{defi}

Une partition localement finie d'un ensemble topologique en variétés s'appelle une \ind{stratification} (voir par exemple \cite{GorMac88} p. 37), ce qui justifie notre terminologie. Les éléments $N_i$ de la partition s'appellent des \ind{strates} ; l'adhérence d'une strate est une variété à bord $\Cl C^1$ par morceaux.

\begin{theo}\label{theo:DomFondStrat}
Soit $(N,g)$ une variété riemannienne, et $p : M\ra N$ un revêtement riemannien galoisien de groupe d'automorphisme $\Gamma$. Soit $f : N\ra \Bb R$ une fonction dont le gradient est simplement stratifié, alors il existe un domaine fondamental $C\subset M$ pour l'action de $\Gamma$ à bord $\Cl C^1$ par morceaux et d'intérieur connexe tel que $\nabla f$ soit tangent à $\bd C$ en dehors des points singuliers.
\end{theo}

\begin{proof}
Soit $f : N\ra \Bb R$ une fonction dont le gradient est simplement stratifié, et $N = \coprod_i N_i$ une stratification associée. Soit $N_1\subset N$ une strate de codimension $0$, comme $N_1$ est simplement connexe, $p$ est trivial au dessus de $N_1$ : les composantes connexes de $p^{-1}(N_1)$ sont toutes difféomorphes à $N_1$. Nous appelons chacune de ces composante connexe un \emph{relevé} de $N_1$ par $p$. Soit $N'_1$ l'un de ces relevés. Le bord de $N'_1$ est constitué d'une union localement finie de strates de codimension au moins $1$. Si $\bar{N_1} = N$, on pose $D = \bar{N'_1}$ : c'est un domaine fondamental pour l'action de $\Gamma$ à bord $\Cl C^1$ par morceaux tel que $\nabla f$ est tangent à $\bd C$ en dehors des points singuliers. Si $\bar{N_1}\neq N$, comme la stratification est localement finie, il existe $N_2\neq N_1$ une autre strate de codimension $0$ telle que $\bar{N_1}$ et $\bar N_2$ contiennent une strate commune $N_{12}$ de codimension $1$. Celle-ci se trouve alors totalement incluse dans l'intérieur de $\bar{N_1\cup N_2}$, qui est donc connexe. Soit $N'_2$ un relevé de $N_2$ tel que l'intérieur de $\bar{N'_1\cup N'_2}$ contiennent un relevé de $N_{12}$. Si $\bar{N_1\cup N_2} = N$, alors $D=\bar{N'_1\cup N'_2}$ est un domaine fondamental vérifiant la conclusion du théorème.  Sinon, il existe une strate $N_3$ distincte de $N_1$ et $N_2$ telle que $\bar{N_1\cup N_2}$ et $\bar{N_3}$ contiennent une strate commune $N_{23}$, ce qui nous permet de reproduire le raisonnement précédent. Les strates de codimension $0$ étant dénombrables, en répétant ce procédé un nombre dénombrable de fois, on obtient un domaine fondamental vérifiant les propriétés voulues.
\end{proof}

Le Théorème \ref{theo:DomFondMorse} découle du théorème précédent d'après le résultat suivant :

\begin{theo}[Thom]
Soit $(N,g)$ une variété riemannienne, et $f : N\ra \Bb R_+^*$ une fonction de Morse telle que pour tout $a>0$, 
$$N_a = \left\{x\in N ; f(x)\geq a\right\}$$
soit compact. Alors le gradient de $f$ est simplement stratifié.
\end{theo}

\begin{proof}
Soit $(N,g)$ une variété riemannienne et $f$ une telle fonction de Morse. Les points critiques de $f$ sont isolés, et on montre que pour tout point critique $\alpha$ de $f$ d'indice $r$, la variété stable de $\alpha$ est difféomorphe à $\Bb R^{n-r}$ (ce résultat, dû à Thom, est démontré par exemple dans \cite{AbRob67}, p87). De plus, par hypothèse, pour tout $a>0$, 
$$N_a = \left\{x\in N ; f(x)\geq a\right\}$$
est compact : cela implique que tout point appartient à une (et une seule) variété stable et que l'ensemble de ces variétés stables est localement fini. Il forme donc une stratification de $M$ en sous-variétés simplement connexes.
\end{proof}

La Figure \ref{fig:DomMorse} illustre le théorème \ref{theo:DomFondStrat} dans le cas où la fonction utilisée est de Morse.

\begin{center}\begin{figure}\label{fig:DomMorse}
\includegraphics[width = \textwidth]{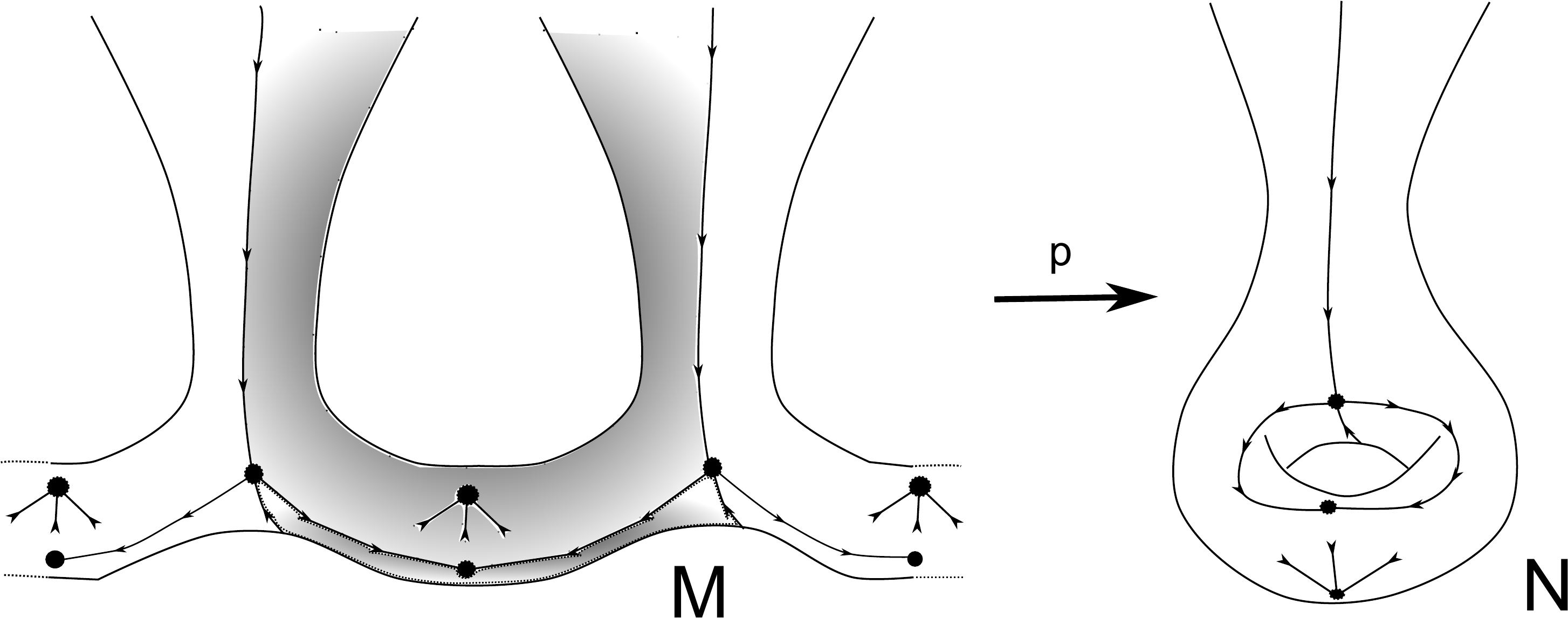}
\caption{\emph{Domaine fondamental et variétés stables :} on a représenté les points critiques de la fonction et leurs variétés stables. En grisé, le domaine fondamental correspondant à la variété stable d'un maximum local}
\end{figure}\end{center}

On peut espérer que des classes plus larges de fonctions vérifient cette propriété de stratification du gradient. Nous disons qu'une variété riemannienne est analytique lorsqu'elle est munie d'un atlas dont les changements de cartes sont analytiques et lorsque sa métrique est une fonction analytique pour la structure définie par cet atlas. On peut construire des fonctions $\Cl C^\infty$ dont le gradient n'est pas simplement stratifié. La question suivante reste à ma connaissance ouverte :
\begin{quest}
Soit $(N,g)$ une variété analytique compacte et $f : M\ra \Bb R$ une fonction analytique. Son gradient est-il simplement stratifié ?
\end{quest}
Nous serons confrontés naturellement par la suite à la question similaire suivante :
\begin{quest}
Soit $(N,g)$ une variété riemannienne $\Cl C^\infty$ dont le trou spectral est strictement positif, et $\phi_0$ sa première fonction propre. Le gradient de $\phi_0$ est-il toujours simplement stratifié ?
\end{quest}

Nous montrons dans \cite{Tap09b} que génériquement, la première fonction propre d'une variété à trou spectral strictement positif est de Morse : son gradient est donc simplement stratifié. Le cas général reste à faire. Nous reviendrions sur ce résultat au Paragraphe \ref{ssec:Brooks}.

\subsection{Hypersurfaces  et théorèmes de comparaison}\label{sec:PrelimGeom}

        Pour établir nos principaux résultats, nous nous placerons au voisinage d'une hypersurface qui sépare deux ouverts. Dans ce paragraphe, nous décrivons les propriétés géométriques et spectrales de ces voisinages dont nous aurons besoin. 
                
        Soit $M$ une variété et $H\subset M$ une hypersurface lisse (éventuellement à bord). On suppose que l'application \ind{exponentielle normale}
        $$\Phi : \left\{\begin{array}{ccc}
                        H\cx[-R,R] & \ra & M\\
                        (x,r)& \mapsto & exp_x(r\nu_x)\end{array}\right.$$
        est un difféomorphisme, où on a noté $\nu_x$ la normale à $H$ en $x$. On dit alors que le \ind{rayon d'injectivité normale} de $H$ est supérieur à $R$, ou que $H$ admet un \ind{voisinage tubulaire} de largeur $R$ dans $M$. On appelle alors $\Phi^*g$  l'écriture de $g$ en \ind{coordonnées de Fermi} au voisinage de $H$. Par exemple, une métrique à courbure sectionnelle constante $-\kappa<0$ au voisinage d'une hypersurface totalement géodésique $H$ s'écrit en coordonnées de Fermi :
        $$\Phi^*g_{\Bb H}(x,r) = (\ch \sqrt{\kappa}r)^2g_H(x)\oplus dr^2,$$
où $g_H$ est la métrique induite sur $H$. La forme volume de $g$ s'écrit alors :
        $$dV_{\Phi^*g}(x,r) = \theta_\kappa^{n-1}(r)dV_H(x)dr,$$ où nous avons noté $\theta_\kappa(r) = \ch (\sqrt \kappa r).$

        De même, soit $(M,g)$ une variété riemannienne quelconque et $H$ une hypersurface ayant un voisinage tubulaire $H_R$ de largeur $R$ dans $M$ :  la forme volume de $g$ s'écrit alors en coordonnées de Fermi sur $H_R$ : $$dV_g(x,r) = \theta^{n-1}(x,r)dV_H(x)dr.$$

Lors de notre démonstration principale, nous étudierons des fonctions définies sur un tel tube $H_R$ ne dépendant que de la distance à $H$. Nous avons alors les résultats suivants :
	
        \begin{prop}[Laplacien d'une fonction radiale]\label{prop:LapRad}
        Soit $(M,g)$ une variété riemannienne et $H$ une hypersurface \emph{compacte} ayant un voisinage tubulaire $H_R$ de largeur $R$ dans $M$. Soit $f= u\circ r : H_R\ra \Bb R$ une fonction ne dépendant que de la distance $r$ à $H$. On a pour tout $r>0$, $$\Delta f = - u''(r)-(n-1)\frac{\theta'(x,r)}{\theta(x,r)}u'(r).$$
        \end{prop}
        
 \begin{proof}
D'après \cite{GHL04} p216, on a $$\Delta r = -(n-1)\frac{\theta'}{\theta},$$ en notant $\theta' = \frac{\bd\theta}{\bd r}.$ On a alors pour toute fonction $u : \Bb R\ra \Bb R$, $$\Delta(u\circ r) = -\gdiv( u'(r)\nabla r) = -u''(r)+\Delta(r) u'(r).$$
\end{proof}

Dans le cas où $M$ est hyperbolique et $H$ totalement géodésique, on retrouve l'expression connue du Laplacien pour une fonction radiale en coordonnées de Fermi :
$$\Delta f = -f''(r) -(n-1)\Th r f'(r).$$

Nous aurons besoin du corollaire suivant :

        \begin{coro}\label{coro:MinLap}
Avec les notations de la proposition précédente, soit $\theta_{\inf} : H_R\ra \Bb R_+^*$ une fonction ne dépendant que de $r$, valant $1$ pour $r = 0$ et telle que $$r\fa \frac{\theta_{\inf}(r)}{\theta(x,r)}$$ soit décroissante pour tout $x\in H$. Pour toute fonction radiale $f = u\circ r$, avec $u$ \emph{décroissante}, et pour tous $(x,r)\in H_R,$ on a alors
    $$\Delta f \geq - u''(r)-(n-1)\frac{\theta_{\inf}'}{\theta_{\inf}}(r)u'(r) =:\Delta_{\inf} f.$$
        \end{coro}
        
        \begin{proof}
Pour tout $x\in H$, on a $\theta(x,0) = \theta_{\inf}(0) = 1$ et $\frac{\theta_{\inf}}{\theta(x,.)}$ est décroissante. On a donc pour tout $(x,r)\in H_R$,
$$\frac{\theta'}{\theta}\geq \frac{\theta_{\inf}'}{\theta_{\inf}}.$$ La proposition précédente conclut notre preuve puisque $u'(r)<0$ pour tout $r>0$.
        \end{proof}
        
\begin{rema}\label{rem:ThetaInf}
Lorsque $H$ est compacte, une telle fonction existe toujours. Il suffit de poser
$$\theta_{\inf}(r) = \exp\left(\int_0^r\inf_{x\in H} \frac{\theta'}{\theta}(x,s) ds\right).$$
On a alors pour tout $(x,r)\in H_R$,
$$\frac{\theta_{\inf}'}{\theta_{\inf}}(r) = (\log\theta_{\inf})'(r) = \inf_{y\in H} \frac{\theta'}{\theta}(y,r)\leq \frac{\theta'}{\theta}(x,r).$$
Lorsque $\theta$ est radial, par exemple lorsque la courbure sectionnelle de $M$ et la courbure moyenne de $H$ sont constantes, on a alors $\theta_{\inf} = \theta$.
\end{rema}

	Notons $V_H(r)$ l'aire de l'hypersurface parallèle à $H$ à distance $r$ :
	$$V_H(r) = \int_H \theta^{n-1}(x,r)dV_H(x).$$
	Nous aurons également besoin par la suite de contrôler une autre quantité géométrique : les \emph{oscillations} de l'élement de volume.
	
	\begin{defi}\label{def:VarVol}
	 Soit $H$ une hypersurface totalement géodésique d'une variété $M$ admettant un voisinage tubulaire $T$ de largeur $R$. Soit $$dV_T = \theta^{n-1}(x,r)dV_H(x)dr$$ son élément de volume. Nous appellerons \ind{fonction d'oscillation} de $\theta$ la fonction
	$$\beta_H(x,r) = \frac{(\theta^{n-1})'(x,r)}{\theta^{n-1}(x,r)} - \frac{V'_H(r)}{V_H(r)}= \frac{(\theta^{n-1})'(x,r)}{\theta^{n-1}(x,r)} - \frac{\int_H (\theta^{n-1})'(x,r)dV_H(x)}{\int_H \theta^{n-1}(x,r)dV_H(x)}.$$
	\end{defi}
	
	Lorsque $\theta$ ne dépend que de $r$, par exemple si $M$ est à courbure sectionnelle constante et si la courbure moyenne de $H$ est constante, on a alors $\beta_H\equiv 0$. En général, pour $H$ compacte, il existe une constante $\Kappa>0$ telle que sur $H_R$, $$|\beta_H|\leq \Kappa.$$
	
	        \subsection{Graphes et moyennabilité}\label{sec:PrelimMoy}
        
Nous présentons maintenant les notions de combinatoire sur les graphes dont nous aurons besoin, pour énoncer et démontrer nos résultats.

	\begin{defi}\label{def:GroupeMoy}
	Soit $G$ un groupe, on dit que $G$ est \ind{moyennable} si et seulement si il admet une \emph{moyenne invariante à gauche}, c'est-à-dire une fonction $m : \Cl P(G)\ra [0,1]$ telle que :
	\bi 
	\item $m(G) = 1 ;$
	\item pour toutes parties $G_1$ et $G_2$ disjointes de $G$, $$m(G_1\cup G_2) = m(G_1)+m(G_2) ;$$
	\item $\forall A\subset G, \forall g\in G, m(gA) = m(A).$
	\ei 
	\end{defi}
Un groupe fini est évidemment moyennable : il suffit d'utiliser la moyenne arithmétique classique. Soit $G = (V,E)$ un graphe. Nous noterons (abusivement) $i\in G$ pour un sommet $i\in V$ de $G$ et $i\sim j$ pour une arête $(i,j)\in E$. 

        \begin{defi}\label{def:GrapheMoy}
 	On appelle \ind{constante de Cheeger} de $G$ la constante
        $$\g h(G) = \inf_{G_f\subset G} \frac{\#\bd G_f}{\#G_f},$$ o\`u
        $G_f$  parcourt l'ensemble des parties finies de $G$ et $\bd G_f$ est  l'ensemble des points de $G_f$ reliés à un point de $G\bs G_f$. Un graphe $G$ est \ind{moyennable} si et seulement si $\g h(G) = 0$
        \end{defi}

	Cette dernière terminologie vient du résultat classique suivant :

        \begin{theo}[F\o lner]
        Soit $\Gamma,S$ un groupe de type fini, $\Gamma$ est moyennable au sens de la Définition \ref{def:GroupeMoy}
        si et seulement si le graphe de Cayley de $\Gamma$
        relativement à tout système fini de générateurs $S$ est moyennable au sens de la Définition \ref{def:GrapheMoy};
        \end{theo}
        On peut par exemple trouver la démonstration de ce théorème dû à F\o lner dans \cite{Broo81CMH}.

\pgh
        Comme cas particuliers des graphes moyennables, citons les graphes finis et les graphes à croissance polynomiale, parmi lesquels les graphes abéliens de type $\Bb Z^n$, et comme exemples de graphes non moyennables, les arbres de valence $\geq 3$. Notons qu'il existe des graphes à croissance exponentielle qui restent moyennables. 

	Sur un graphe $G$, nous appellerons  \ind{Laplacien combinatoire} l'opérateur $\Delta_G$ défini par $$(\Delta_G f)(i) = \sum_{i\sim j} (f(i) - f(j)).$$
        C'est un opérateur autoadjoint positif, dont le bas du spectre $\mu_0(G)$ est l'infimum des quotients de Rayleigh combinatoires :
        $$\mu_0(G) = \inf_{f}\frac{\sum_{i\sim j}\left(f(i)-f(j)\right)^2}{\sum_i f(i)^2}$$
où $f$ parcours l'ensemble des fonctions à support compact dans $G$. Le théorème suivant relie la moyennabilité du graphe au bas du spectre de cet opérateur :
\begin{theo}[Inégalités de Cheeger combinatoires]\label{theo:CheegComb}
Pour tout graphe $G$ de valence constante $v$, on a
        $$\frac{1}{2v}\g h(G)^2\leq \mu_0(G)\leq \g h(G).$$
        \end{theo}
En particulier, le graphe est moyennable si et seulement si le bas du spectre de son Laplacien combinatoire est
        nulle. On peut consulter \cite{Colin98} p31 pour une preuve de ce résultat.	

\section{Définitions, exemples et résultats}

        \subsection{Variétés $G$-périodiques}\label{ssec:GPeriod}
        
Soit $G = (V,E)$ un \ind{graphe}, où $V$ désigne l'ensemble des sommets de $G$ et $E\subset V\cx V$ l'ensemble des arêtes. Pour $i,j\in V$, nous noterons en général $i\sim j$ au lieu de $(i,j)\in E$, et $i\in G$ au lieu de $i\in V$.
        
        \begin{defi}\label{def:GPeriod}        
        Soit $G$ un graphe à valence constante $v<\infty$. Nous dirons que $M$ est une \indb{variété}{$G$-périodique} s'il existe une 							sous-variété à bord $\Cl C^1$ par morceaux $C\subset M$ fermée d'intérieur non vide, qui vérifie les propriétés suivantes.
        \bi
        \item Il existe des sous-variétés $C_i\subset M$ d'intérieurs disjoints telles que $$ M = \bigcup_{i\in G} C_i,$$
        et pour tout $i\in G$, il existe une isométrie $J_i : C_i\ra C$.
        \item Pour tous $i,j\in G, i\neq j$, si $i\sim j$ alors le bord commun $C_i\cap C_j$ contient une sous-variété de codimension $1$.
        \item Il existe $R>0$ et une famille d'ouverts $\alpha_1,...,\alpha_v\subset \bd C$, pas forcément connexes, d'adhérences compacte dans $\bd C$, de classe $\Cl C^1$, possédant des voisinages tubulaires $T_1,...,T_v$ de rayon $R$ disjoints dans $C$, tels que si $i\sim j$, on ait les propriétés :
        \bi
        	\item la composante de bord $J_i(C_i\cap C_j)\subset \bd C$ contient l'un des $\alpha_k$ ; de même $J_j(C_i\cap C_j)$ contient $\alpha_{k'}$ ;
        	\item l'application $$J_i\circ J_j^{-1} : J_j(C_i\cap C_j)\ra J_i(C_i\cap C_j)$$ induit une isométrie de $\alpha_{k'}$ sur $\alpha_k$.
        \ei
        \ei
        Nous appelons $C$ la \ind{cellule} élémentaire de $M$, $G$ le \indb{graphe}{modèle} de $M$, et les $\alpha_k$ les \ind{zones de transition}.
        \end{defi}

Une cellule de valence $v$ est donc la donnée d'une variété à bord $\Cl C^1$ par morceaux, et de $v$ zones de transition. Notons que lorsque $\bd C$ est compact, les deux premières hypothèses de la définition impliquent les suivantes. Lorsque $\bd C$ est non compact, il est possible d'avoir un découpage d'une variété $M$ en une famille de variétés à bord toutes isométriques entre elles, mais qu'il n'existe aucun graphe $G$ localement fini pour lequel $M$ soit $G$-périodique. L'expression \emph{<<Soit $C$ une cellule>>} signifie donc que sont fixés :
\bi
\item une variété $C$ à bord $\Cl C^1$ par morceaux ;
\item une valence $v>0$ et des composantes de bords compactes disjointes $\alpha_1,...,\alpha_v\subset \bd C$ qui admettent des voisinages tubulaires $T_1,...,T_k\subset C$ de largeur $R>0$ disjoints.
\ei
        
                Un exemple simple d'une telle variété est le pavage régulier du plan : dans ce cas, $C$ est un carré, et $G = \Bb Z^2$ :
                
\begin{center}
\includegraphics[width = 0.5\textwidth]{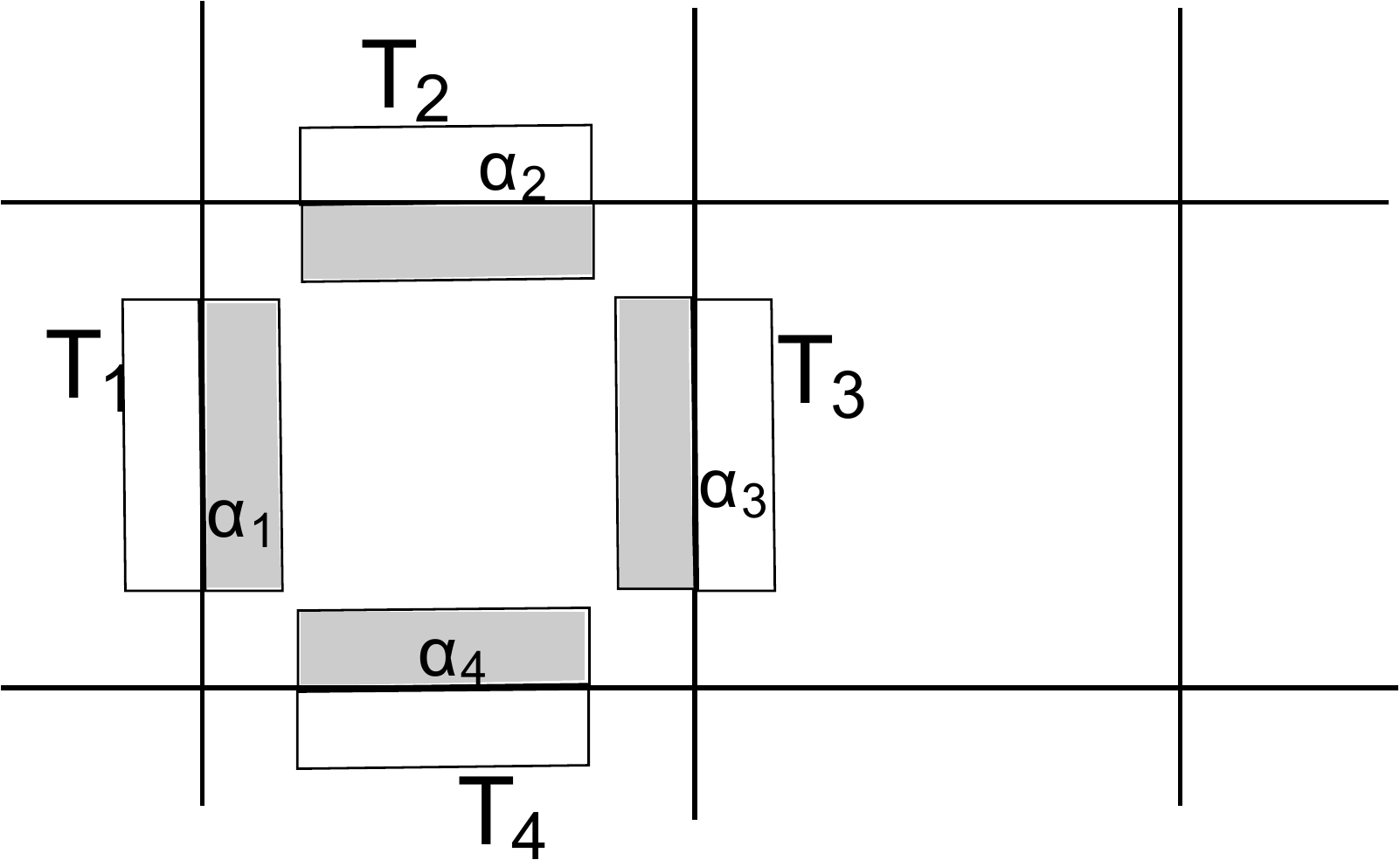}

\emph{Cellule et zones de transition pour le pavage régulier du plan}
\end{center}

Ces variétés sont une généralisation naturelle des \emph{variétés périodiques} usuelles, modelées en général sur $\Bb Z^n$. Le groupe fondamental d'une variété $G$-périodique n'est pas de type fini dès que le graphe $G$ est infini et le groupe fondamental de la cellule relativement au bord $\pi_1(C)/\pi_1(\bd C)$ est non trivial. La cellule élementaire d'une variété $G$-périodique n'est jamais unique ; nous verrons par la suite que la question clé pour obtenir des résultats spectraux est souvent de \emph{choisir une cellule adaptée} au problème considéré.

\pgh
On peut obtenir une vaste famille d'exemples en recollant des surfaces hyperboliques le long de géodésiques. Soit $C$ une surface hyperbolique dont le bord est la réunion de $v$ géodésiques fermées, toutes de même longueur. Pour tout graphe $G$ de valence $v$, il existe une surface $M$ modelée sur $G$ à partir de $C$, construite par recollement de copies de $C$ le long de leurs bords. La métrique hyperbolique ainsi obtenue est lisse (voir \cite{BePe92}, chapitre B). Les géodésiques de bord forment ici les zones de transition, et elles admettent un voisinage tubulaire dont une largeur minimale est donnée par le Lemme du Collier (voir par exemple \cite{Colb85}). On obtient en particulier de nombreuses surfaces hyperboliques topologiquement infinies, qui généralisent les surfaces périodiques. La Figure \ref{fig:SurfGraphes} présente deux exemples de cette sorte.

\begin{center}\begin{figure}
\includegraphics[width = 0.5\textwidth]{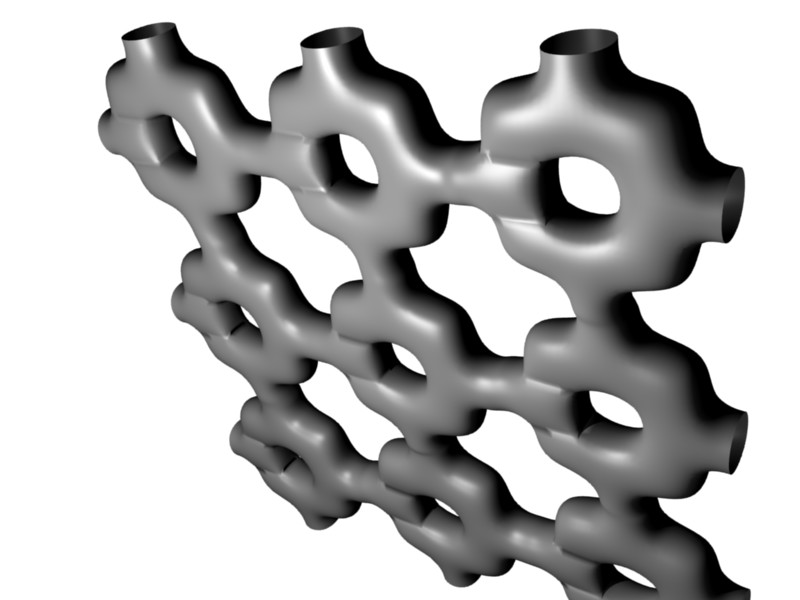}\includegraphics[width = 0.5\textwidth]{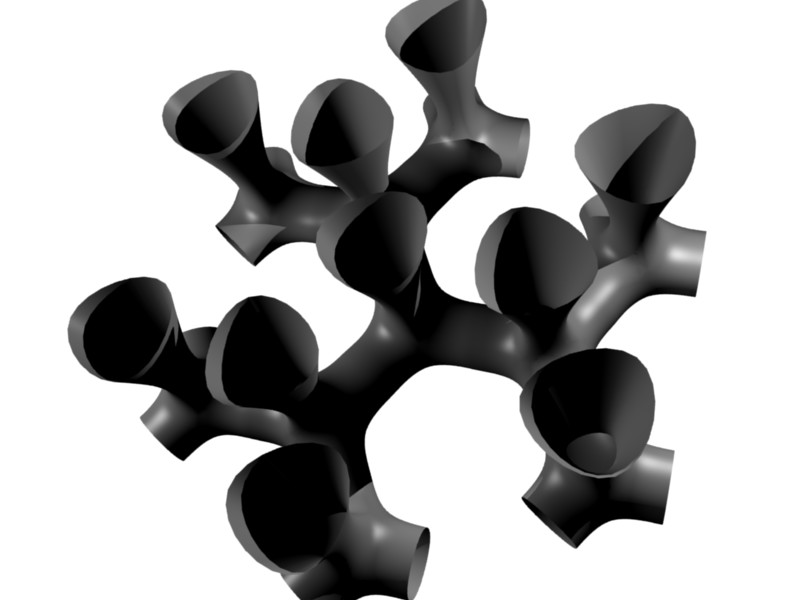}
\caption{Surfaces modelées sur $\Bb Z^2$ à partir d'un tore à $4$ composantes de bord et sur l'arbre de degré 3 à partir d'une cellule de volume infini}\label{fig:SurfGraphes}
\end{figure}\end{center}

	Etant donnés une cellule $C$ et un graphe $G$, il n'y a pas nécessairement unicité à isotopie près de la surface $G$-périodique que l'on peut construire de cette façon : des twists de Dehn le long des géodésiques recollées restent possibles.
\pgh
   	
	Une autre grande famille dde variétés topologiquement infinies à laquelle nos résultats s'appliquent est celle des \ind{revêtements riemanniens galoisiens}. Soit $\pi : M\ra N$ un revêtement riemannien galoisien, de groupe d'automorphisme $\Gamma$ de type fini. Nous avons vu à la Section \ref{sec:CoinsRevet} qu'il existe un domaine fondamental $C$ dans $M$ pour l'action de $\Gamma$ dont le bord est $\Cl C^1$ par morceaux. Soit $S_C = \{\gamma_1,...,\gamma_v\}$ un système de générateurs (symétrique) asssocié au domaine $C$ par la Proposition \ref{prop:GeneGroupeRev}. L'hypersurface ($\Cl C^1$ par morceaux) $C\cap \gamma_1(C)\subset \bd C$ est de codimension 1, $\Cl C^1$ par morceaux. En particulier, elle contient un ouvert (non vide) d'adhérence compacte $\alpha_1$ lisse, qui admet un voisinage tubulaire $T_1$ de largeur $R$. Supposons que $\gamma_1^{-1} = \gamma_{\frac{v}{2}+1}$, on pose $\alpha_{\frac{v}{2}+1} = \gamma^{-1}_1(\alpha_1)$. Quitte à diminuer $R$, on peut supposer que le voisinage tubulaire $T_{\frac{v}{2}+1}$ de $\alpha_{\frac{v}{2}+1}$ de largeur $R$ n'intersecte pas $T_1$. Puis, si $\gamma_2\notin\{\gamma_1,\gamma_1^{-1}\},$ il existe un ouvert $$\alpha_2\subset C\cap\gamma_2(C)\subset \bd C$$ de codimension 1, d'adhérence compacte lisse dans $\bd C$, et quitte à rétrécir $\alpha_1$ et $R$, on peut supposer que son voisinage tubulaire de largeur $2R$ n'intersecte pas $T_1$ et $T_{\frac{v}{2}+1}$. En répétant ce procédé un nombre fini de fois, on obtient une cellule qui vérifie les propriétés demandées à la Définition \ref{def:GPeriod} et une variété $G$-périodique, où $G$ est le \ind{graphe de Cayley} de $\Gamma$ associé au système de générateurs $S_C$.

Le pavage régulier du plan correspond au revêtement canonique $p : \Bb R^2 \ra \Bb T^2$. La Figure \ref{fig:DomMorse} présente un exemple de variété $\Bb Z$-périodique où la cellule a un bord non-compact. On fixe les zones de transition en choisissant un compact dans chaque composante de bord, d'intérieur non vide.

Dans l'étude spectrale de ces variétés $G$-périodiques, un premier constat est immédiat :
\begin{prop}
Soit $G$ un graphe à valence constante $v<\infty$, et $M$ une variété $G$-périodique de cellule $C$. Alors
$$\lambda_0(M)\geq \lambda_0(C),$$
où l'on impose les conditions de Neumann sur $\bd C$.
\end{prop}
\begin{proof} C'est un corollaire direct de la Proposition \ref{prop:MonoNeum}.\end{proof}

        \subsection{Recollement de fontions propres}

        Soit $G$ un graphe de valence constante $v$, et $M$ une variété $G$-périodique de cellule $C$. On suppose désormais que $C$ admet une fonction propre $\phi_0$ pour son bas du spectre $\lambda_0$ (on rappelle que sur une variété à bord nous considérons toujours le problème de Neumann, défini à la Section \ref{ssec:PreSpec}). Pour $i\in G$, on note toujours $J_i : C_i\ra C$ l'isométrie d'identification de la cellule $C_i$.
        
	\begin{defi}
	Nous dirons que $\phi_0$ \emph{se recolle bien dans $M$} si la fonction $\widetilde\phi_0 : M\ra \Bb R$ définie sur l'intérieur de chaque cellule $C_i$ par
	$$\widetilde\phi_0|_{\inter{C_i}} = \phi_0\circ J_i$$ se prolonge en une application continue sur $M$ .
	\end{defi}
	
	Nous appellerons $\widetilde\phi_0$ \emph{l'extension de $\phi_0$ à $M$}, et nous la noterons souvent simplement $\phi_0$. Cette hypothèse de recollement de la première fonction propre sera cruciale pour la partie délicate des démonstrations de nos résultats. Il faut la voir comme une hypothèse \emph{sur la cellule $C$}. Elle est certes très restrictive ; mais pour nombre de variétés $G$-périodiques intéressantes, il existe une cellule pour laquelle elle est vérifiée. Reprenons les deux familles d'exemples décrites à la section précédente.
	
\pgh
Dans le cadre des surfaces hyperboliques, une hypothèse supplémentaire de \emph{symétrie globale de la cellule} assure que la fonction propre se recolle bien. 
	
	Soit $C$ une surface hyperbolique dont le bord est constitué de $v$ géodésiques fermées que nous notons $\alpha_1,..., \alpha_v$. Supposons que $C$ est munie d'une isométrie globale $\Cl I$ d'ordre $v$ telle que pour tout $k = 1,...,v$, 
	$$\Cl I \alpha_k = \alpha_{k+1},$$
	en notant $\alpha_{v+1} = \alpha_1$. Soit $G$ un graphe à valence $v$ et $\{J_i : C_i\ra C\}_{i\in G}$ des copies de $C$. Soit $M$ la surface modelée sur $G$ à partir de $C$ définie en recollant $\alpha_k^i$ à $\alpha_{k'}^j$ via $J_i^{-1}\circ\Cl I^{k-k'}\circ J_j$, selon n'importe quelle combinaison des bords des cellules telle que la surface ainsi obtenue respecte la combinatoire du graphe. 
		
	\begin{prop}
	Si $C$ admet une première fonction propre $\phi_0$, alors celle-ci se recolle bien dans $M$.
	\end{prop}
	\begin{proof}
	Il suffit de remarquer que $\phi_0$ est invariante par $\Cl I$ d'après la proposition \ref{prop:InvEigFunc}.
	\end{proof}

\pgh
Dans le cas des revêtements, l'hypothèse de recollement est assurée par un critère très simple à vérifier :

\begin{prop}
Soit $p : M\ra N$ un revêtement galoisien de groupe d'automorphisme $\Gamma$ de type fini. On suppose qu'il existe une première fonction propre $\phi_0 : N\ra \Bb R$ associée au bas du spectre $\lambda_0(N)$. 

S'il existe un domaine fondamental à coins $C\subset M$ tel que $\phi_0\circ p$ vérifie les conditions de Neumann sur $\bd C$, alors $$\lambda_0(C) = \lambda_0(N)$$ et la première fonction propre de $C$ existe et est donnée par la restriction de $\phi_0\circ p$ à  $C$. Elle s'étend donc par l'action de $\Gamma$ en une fonction $\widetilde{\phi_0}$ continue sur $M$.
\end{prop}

\begin{proof}
Comme $\phi_0$ est $\Cl H^1$ sur $N$, son relevé à $C$ l'est également. Elle est aussi $\lambda_0(N)$-harmonique, vérifiant les conditions de Neumann sur $\bd C$ par hypothèse : elle est donc par définition valeur propre du Laplacien avec condition de Neumann associée à $\lambda_0(N)$, et puisqu'elle est positive, $\lambda_0(N)$ est nécessairement le bas du spectre de $C$.
\end{proof}

Nous avons vu à la Section \ref{sec:CoinsRevet} une condition suffisante pour qu'un tel domaine fondamental existe : 
\begin{prop}
Soit $(N,g)$ une variété riemannienne qui admet une première fonction propre $\phi_0 : N\ra \Bb R$ associée au bas du spectre $\lambda_0(N)$. Si le gradient de la première fonction propre $\phi_0$ est \emph{simplement stratifié}, alors pour tout revêtement $p : M\ra N$ de groupe d'automorphisme $\Gamma$, il existe un domaine fondamental à coins $C\subset M$ tel que $\phi_0\circ p$ vérifie les conditions de Neumann sur $\bd C$.
\end{prop}
Il s'agit d'un corollaire direct du Théorème \ref{theo:DomFondStrat}.

Nous nous sommes demandés à la Section \ref{sec:CoinsRevet} si la première fonction propre avait toujours un gradient simplement stratifié : en cas de réponse positive, la proposition précédente s'étend à toutes les variétés riemanniennes à trou spectral strictement positif. Dans \cite{Tap09b}, nous montrons que pour une métrique générique, c'est le cas : la première fonction propre est de Morse.

\begin{theo}[\cite{Tap09b}]\label{th:ResumDomFond}
Soit $(N,g)$ une variété riemannienne de volume infini de rayon d'injectivité et de trou spectral strictement positifs, dont les dérivées partielles du tenseur de courbure d'ordre inférieur à $\frac{n}{2}$ sont uniformément bornées. Soit $\phi_0$ sa première fonction propre, pour tout $a>0$, les ensembles $\{x\ ; \phi_0(x)>a\}$ sont compact. De plus, lorsque la métrique est \ind{générique}, $\phi_0$ est de Morse. Pour tout revêtement riemannien $p : M\ra N$ de groupe d'automorphisme $\Gamma$, il existe alors un domaine fondamental pour $\Gamma$ sur lequel $\phi_0$ vérifie les conditions de Neumann.
\end{theo}

Dans ce théorème, le terme \emph{<<générique>>} est employé au sens classique suivant : pour toute métrique $g$, il existe un voisinage $\Cl V_g$ pour une topologie adaptée à l'ensemble des métriques sur une variété non compacte (la \ind{topologie forte}, ou \ind{topologie de Whitney}) tel que toutes les métriques à l'intérieur de ce voisinage ont un trou spectral strictement positif, et l'ensemble de celles dont la première fonction propre est de Morse est une intersection dénombrable d'ouverts denses. Nous invitons le lecteur à se référer à \cite{Tap09b} et aux références qui y sont donnés pour plus de détails sur ce résultat et les topologies utilisées.

Nous travaillerons donc désormais en supposant que nous avons une variété $G$-périodique et une cellule $C$ dont la première fonction propre se recolle bien. Nous venons de voir que dans de nombreuses situations intéressantes, cette hypothèse est vérifiée !

	\subsection{Résultats : contrôle du bas du spectre}\label{ssec:Result}
	
	Notre premier résultat est une adaptation au cadre des variétés $G$-périodiques de méthodes classiques reliant moyennabilité et bas du spectre :
	
	\begin{theo}\label{theo:MajoGene}
	Soit $G$ un graphe à valence constante et $M$ une variété $G$-périodique dont la cellule $C$ admet une première fonction propre qui se recolle bien dans $M$. Si $G$ est moyennable, on a
	$$\lambda_0(M) = \lambda_0(C).$$
	Si $G$ n'est pas moyennable, si le rayon d'injectivité normale de $\bd C$ est strictement minoré par $R_{min}$, et si le nombre de voisins de chaque cellule (c'est-à-dire le nombre d'autres cellules que son bord intersecte) est une constante $k<\infty$, alors il existe une constante $A'>0$  qui ne dépend que de $k$, de $R_{min}$ et de $\lambda_0(C)$ telle que
	$$\lambda_0(M)\leq \lambda_0(C) + A'\g h(G).$$
	\end{theo}
		
	Notre principal résultat lorsque $G$ n'est pas moyennable est le suivant :
	
	\begin{theo}\label{theo:ControleGene}
	Soit $C$ une cellule de valence $v$ qui admet une première fonction propre $\phi_0$ et un trou spectral $\eta = \lambda_1(C)-\lambda_0(C)>0$. Alors il existe une constante $A>0$, ne dépendant que de $\lambda_0$, des zones de transition de $C$ et valeurs de $\phi_0$ sur $\bd C$, telle que pour tout graphe $G$ de valence $v$ et toute variété $G$-périodique $M$ de cellule $C$ où $\phi_0$ se recolle bien,
	$$\lambda_0(M)\geq \lambda_0(C) + A\eta\mu_0(G).$$
	\end{theo}
	
	Les constantes $A$ et $A'$ restent positives et bornées lorsque $\lambda_0(C)$ varie ; elles dépendent principalement de la géométrie sur les zones de transition de $C$ et des valeurs de $\phi_0$ sur $\bd C$. Cette dépendance sera détaillée au cours de la démonstration. On a vu que l'on a toujours
	$$\lambda_0(M)\geq \lambda_0(C).$$ Si l'on ne cherche pas à garder le contrôle explicite de l'écart entre les bas de spectres, ces théorèmes deviennent alors :
	
	\begin{coro}\label{coro:InegMoy}
	Soit $C$ une cellule de valence $v$ qui admet une première fonction propre $\phi_0$ et un trou spectral strictement positif, soit $G$ un graphe de même valence $v$, et $M$ une variété modelée sur $G$ à partir de $C$ où $\phi_0$ se recolle bien. Alors
	$$\lambda_0(M)\geq \lambda_0(C),$$
	avec égalité si et seulement si $G$ est moyennable.
	\end{coro}
	
Appliqué au cas des revêtements qui admettent un domaine fondamental dont la première fonction propre se recolle bien, ce corollaire et le Théorème \ref{th:ResumDomFond} nous permettent d'obtenir le corollaire suivant :
	
	\begin{coro}
Soit $N$ une variété riemannienne dont la courbure de Ricci et le rayon d'injectivité sont uniformément minorés, et le trou spectral strictement positif, soit $\phi_0$ sa première fonction propre. Soit $p : M\ra N$  un revêtement riemannienn galoisien, dont le groupe d'automorphisme $\Gamma$ est de type fini. S'il existe un domaine fondamental pour l'action de $\Gamma$ dans $M$ sur lequel le relevé de $\phi_0$ vérifie les conditions de Neumann, alors
	$$\lambda_0(M)\geq \lambda_0(N)$$
	avec égalité si et seulement si $\Gamma$ est moyennable. C'est le cas dès que les $\frac{n}{2}$ premières dérivées du tenseur de courbure sont uniformément bornées et que la métrique est générique.
	\end{coro}
		
	De nouveau, notre méthode permet de préciser ce corollaire en donnant des constantes explicites qui bornent $\abs{\lambda_0(M)-\lambda_0(C)}$, qui dépendent alors de la combinatoire de $\Gamma$, de la géométrie au bord du domaine fondamental, et des valeurs de $\phi_0$ sur ce bord. Nos résultats peuvent être vus comme une généralisation des travaux de R. Brooks \cite{Broo85Reine}, qui montre que pour un revêtement $p : M\ra N$, sous certaines hypothèses géométriques qui impliquent entre autres que le trou spectral de $N$ est strictement positif,	$\lambda_0(M)\geq \lambda_0(N)$	avec égalité si et seulement si $\Gamma$ est moyennable. Nous détaillerons au paragraphe \ref{ssec:Brooks} le lien entre nos résultats et ceux de Brooks.

\section{Démonstrations principales}\label{sec:Demo}

        \subsection{Majoration du $\lambda_0$ et graphes moyennables}

	Soit $G$ un graphe de valence $v$ et $M$ une variété $G$-périodique, de cellule $C$. On note  $\lambda_0$ le bas du spectre du spectre $C$, et on suppose qu'elle admet une première fonction propre $\phi_0$ normalisée qui se recolle bien dans $M$. L'objectif de cette section est de démontrer le Théorème \ref{theo:MajoGene}.	
		
	\begin{proof}
	Rappelons que d'après la Proposition \ref{prop:MonoNeum}, on a toujours $\lambda_0(M)\geq \lambda_0.$
	
	Soit $\phi_\epsilon\in\Cl H^1(M)$ une fonction positive lisse à support compact $K\subset C$ telle que $$\frac{\norm{\nabla \phi_\epsilon}^2}{\norm{\phi_\epsilon}^2}\leq \lambda_0(C) + \epsilon,$$ et qui se recolle bien dans $M$. On peut la construire par exemple en prenant la restriction de $\phi_0$ à un compact $K_\epsilon\subset K$ et en lissant la fonction ainsi obtenue de façon à la rendre $\Cl C^\infty$ à support dans $K$. 
	On supposera pour simplier les notations que $\norm{\phi_\epsilon}^2 = 1$. Pour toute cellule $C_i\subset M$ donnée par la Définition \ref{def:GPeriod}, et $J_i : C_i\ra C$ son isométrie d'identification, on note $$K_i\subset C_i = J_i^{-1}(K).$$ Soit $R>0$ fixé, alors pour tout $i\in G$, un voisinage de $K_i$ de largeur $R$ intersecte un nombre fini des $K_j, j\neq i$ : nous noterons ce nombre $k_\epsilon$, qui ne dépend pas de $i$.
	
	Soit $(G_p)_{p\in\Bb N}$ une famille de parties finies (que nous pouvons supposer connexes) de $G$ telle que 
	$$\lim_{p\ra \infty}\frac{\#\bd G_p}{\#G_p}=\g h(G).$$
	On note $M_p$ la réunion des cellules correspondant à la partie $G_p$.
	
	On note $M_p^+$ un voisinage de largeur $R$ de $M_p$ : l'ensemble $M_p^+\bs M_p$ intersecte donc au plus $k_\epsilon(\#\bd G_p)$ cellules.  Nous allons simplement reproduire $\phi_\epsilon$ sur toutes les cellules de $M_p$, et rendre la fonction obtenue continue à support compact dans $M_p^+$. Cette dernière opération ajoutera au quotient de Rayleigh un terme de l'ordre de $k_\epsilon\frac{\#\bd G_p}{\# G_p}$.
	
	\pgh

On définit sur $M_p^+\bs M_p$ la fonction
$$\psi_p(x) = \frac{1}{R}(R-d(x,M_p))\leq 1.$$ On l'étend en une fonction que nous notons toujours $\psi_p : M\ra \Bb R$ qui vaut 1 sur $M_p$ et $0$ sur $M\bs M_p^+$. Comme $\bd M_p$ est $\Cl C^1$ par morceaux, $\psi_R$ est $\Cl C^1$ sur $M\bs \bd M_p$, et continue sur $\bar{M\bs \bd M_p}$. On sait qu'alors, $\forall x\in M_p^+\bs M_p$, on a
$$|\nabla \psi_p| = \frac{|\nabla d(.,M_p)|}{R}\leq \frac{1}{R}.$$

Notons encore $\phi_\epsilon$ l'extension de $\phi_\epsilon$ à $M$, qui existe et est continue par hypothèse. 

Soit $\widetilde{\phi_p}$ la fonction définie sur $M$ par	$\widetilde{\phi_p}= \phi_\epsilon\cdot\psi_p.$
La fonction $\widetilde{\phi_p}$ est $\Cl C^1$ par morceaux, continue sur $M$ et à support compact dans $M_p^+$. Calculons son énergie.
	
Pour tout $i\in G_p$, on a $$\norm{\nabla \widetilde{\phi_p}}^2_{C_i} = \int_{C_i} |\nabla \widetilde{\phi_p}|²=  \int_C |\nabla \phi_\epsilon|²\leq \lambda_0+\epsilon.$$

Soit maintenant $C_i$ une cellule qui intersecte $M_p^+\bs M_p$. Sur $C_i$, la fonction $\psi_p$ est positive, inférieure à $1$ et son gradient est de norme $1/R$, d'où
	$$\int_{C_i}|\nabla \widetilde{\phi_p}|² = \int_{\alpha_R^+}|\phi_\epsilon\nabla \psi_p+\psi_p\nabla\phi_\epsilon|²\leq \frac{1}{R^2}\int_C|\phi_\epsilon|² + 2\int_C|\phi_\epsilon g(\nabla\phi_\epsilon,\nabla (\psi_p\circ J_i^{-1}))| + \int_C|\nabla\phi_\epsilon|²$$
	$$\leq \frac{1}{R^2} + 2\frac{\sqrt{\lambda_0}}{R} + \lambda_0 + \epsilon.$$
On a obtenu la dernière majoration à l'aide de l'inégalité de Cauchy-Schwarz.

	On a donc $$\norm{\nabla \widetilde{\phi_p}}^2_M = \sum_{i\in G_p} \norm{\nabla \widetilde{\phi_p}}^2_{C_i} + \sum_{C_i\subset M_p^+\bs M_p} \norm{\nabla \widetilde{\phi_p}}^2_{C_i},$$
	soit d'après ci-dessus
	$$\norm{\nabla \widetilde{\phi_p}}^2_M \leq \#G_p (\lambda_0+\epsilon) + \#\left\{C_i\subset \bd M_p^+\bs M_p\right\}\left(\frac{1}{R^2}+\frac{2\sqrt\lambda_0}{R}+\lambda_0+\epsilon\right)$$
	$$\leq \#G_p (\lambda_0+\epsilon) + k_\epsilon\#(\bd G_p)\left(\frac{1}{R^2}+\frac{2\sqrt\lambda_0}{R}+\lambda_0+\epsilon\right).$$
	
	De plus, $\norm{\widetilde{\phi_p}}^2_M\geq \norm{\widetilde{\phi_p}}^2_{M_p} = \#G_p.$ On a donc
	$$\frac{\norm{\nabla \widetilde{\phi_p}}^2_M}{\norm{\widetilde{\phi_p}}^2_M}\leq \frac{\#G_p (\lambda_0 +\epsilon)+ k_\epsilon\#(\bd G_p)\left(\frac{1}{R^2}+\frac{2\sqrt\lambda_0}{R}+\lambda_0+\epsilon\right)}{\# G_p} = \lambda_0 +\epsilon + (A'+\epsilon k_\epsilon)\frac{\#\bd G_p}{\#G_p},$$
	en posant $$A' = k_\epsilon\left(\frac{1}{R^2}+\frac{2\sqrt\lambda_0}{R}+\lambda_0\right).$$
	
	Lorsque $G$ est moyennable, on a
	$$\lim_{p\ra\infty}\frac{\#\bd G_p}{\#G_p} = 0.$$
	La majoration précédente implique donc le premier point du Théorème \ref{theo:MajoGene}. Lorsque $C$ n'a qu'un nombre fini $k$ de voisins, et le rayon d'injectivité normale du bord est minoré par $R_{min}>0$, on peut prendre $R = R_{min}$ et $k_\epsilon = k$ pour tout $\epsilon>0$ . Ceci conclut notre preuve par passage à la limite lorsque $p\ra\infty$ et $\epsilon\ra 0$.
	
	\end{proof}
 
	Cette partie de la démonstration s'adapte dans le cas où $C$ n'admet pas de première fonction propre : il suffit qu'il existe une suite de fonction $(\phi_\epsilon),_{\epsilon>0}$ dont le quotient de Rayleigh tend vers $\lambda_0$, qui se recollent toutes dans $M$. Nous n'avons pas besoin que ces fonctions vérifient les conditions de Neumann sur $\bd C$. 
	
	Dans le cas d'un revêtement $p : M\ra N$ de groupe d'automorphisme $\Gamma$ et domaine fondamental $C$, la démonstration est analogue (même lorsque $\lambda_0(N)$ n'est pas le bas du spectre de $C$). On choisit une fonction $\phi_\epsilon$ sur $N$ à support compact dont le quotient de Rayleigh est inférieur à $\lambda_0+\epsilon$ et une famille de parties finies $G_p\subset G$ où $G$ est un graphe de Cayley de $\Gamma$, qui vérifie 	$$\lim_{p\ra\infty}\frac{\#\bd G_p}{\#G_p} = \g h(\Gamma).$$
	On relève $\phi_\epsilon$ à la réunion des domaines fondamentaux correspondant à $G_p$, et on la rend à support compact de la même façon que dans la démonstration précédente : une majoration identique à celle du Théorème \ref{theo:MajoGene} s'ensuit. Notre démonstration est très proche de celle de \cite{Broo81CMH}, §2 et a été utilisée depuis dans de nombreux articles traitant de moyennabilité.

	\begin{rema}
	Dans le cas où $G$ est le graphe de Cayley d'un groupe abélien (par exemple $\Bb Z, \Bb Z^2...$), donc moyennable, on retrouve ce qui est appelé communément une \indb{variété}{périodique}. Le résultat $\lambda_0(M) = \lambda_0(C)$ est alors un corollaire de la \ind{théorie de Floquet}. La démonstration précédente est plus élémentaire que la construction de la théorie de Floquet, et s'applique à une classe de variétés beaucoup plus large (mais le résultat est beaucoup plus faible). 
	\end{rema}
	
        \subsection{Graphes non moyennables et minoration du $\lambda_0$}

Nous allons maintenant nous intéresser à notre résultat principal, la minoration du Théorème \ref{theo:ControleGene}, qui n'a d'intérêt que lorsque $G$ n'est pas moyennable. 
	
	Soit $C$ une cellule de valence $v$, et $\alpha_1,...,\alpha_v\subset C$ ses zones de transition. Par hypothèse, les $\alpha_i$ admettent des voisinages tubulaires de largeur $R$ disjoints dans $C$ qui s'écrivent $T_i = \alpha_i\cx [0,R].$ On note toujours $\lambda_0$ le bas du spectre (avec condition de Neumann) de $C$, on suppose qu'elle admet une première fonction propre $\phi_0$ normalisée et que le trou spectral $\eta = \lambda_1-\lambda_0>0$ (voir Définition \ref{def:Trou}). Soit $G$ un graphe de valence $v$, et $M$ une variété $G$-périodique de cellule $C$, telle que $\phi_0$ se recolle bien dans $M$. L'objectif de cette section est donc de démontrer :
	
	\begin{theo}\label{th:MinNMoy}
	Avec les notations précédentes, on a
	$$\lambda_0(M)\geq\lambda_0^N(C) +A\eta\mu_0(G),$$
	où $A>0$ dépend  de $\lambda_0$, de la géométrie des zones de transition, et des valeurs de $\phi_0$ sur ces zones de transition.
	\end{theo}
	
	
	Notons $\tilde \lambda_0 = \lambda_0(M)$ et $\delta = \tilde\lambda_0-\lambda_0.$	D'après la Proposition \ref{prop:MonoNeum}, $\delta\geq 0$.
	
	La structure de la démonstration de ce théorème, un peu longue, est la suivante. Nous prenons une fonction test sur $M$ dont nous voulons évaluer le quotient de Rayleigh. Nous la discrétisons en la projetant, sur chaque cellule, sur la fonction propre $\phi_0$, et notre théorème s'obtient par une minoration du terme d'erreur qui reste après cette projection. Nous effectuons cette minoration sur les tubes de jonctions qui s'appuient sur les zones de transition entre deux cellules. Là, nous montrons que deux possibilités se présentent : soit le terme d'erreur a une norme $L^2$ minorée, soit la norme de son gradient est minorée, et dans les deux cas cela nous permet de nous ramener au spectre du Laplacien combinatoire sur le graphe et de conclure.
	
	\begin{rema}	
	L'idée de discrétiser une variété pour obtenir un controle sur le spectre du Laplacien riemannien à partir du Laplacien combinatoire sur le graphe sous-jacent à la discrétisation n'est pas nouvelle. Elle est à la base des travaux de R. Brooks que nous avons déjà mentionnés, et plus près de nous de ceux, on peut citer ceux de T. Coulhon et L. Saloff-Coste (voir \cite{CoulSal95}) et de T. Mantuano (voir \cite{Man05}). Dans ces deux articles, on discrétise la variété grace à un recouvrement localement fini par des boules, et à chaque fonction $f$ sur la variété, on associe une fonction sur le graphe qui a pour valeur en un sommet la moyenne de $f$ sur la boule correspondante (en pondérant éventuellement le graphe par le volume de la boule). Ici, la discrétisation se fait en projetant sur des <<morceaux>> de volume infini : les cellules, que nous pondérons à l'aide de leur fonction propre pour rendre ces projections finies. 
	\end{rema}
	
	\begin{proof}[Démonstration du Théorème \ref{th:MinNMoy}]
	Soit $(f_\epsilon)_{0<\epsilon<1}$ une famille de fonctions à support compact dans $M$ telle que pour tout $\epsilon>0$, on ait
	$$\frac{\int_M|\nabla f_\epsilon|^2}{\int_M f_\epsilon^2}\leq \tilde\lambda_0 + \epsilon.$$
	
	Soit $\epsilon>0$, nous discrétisons $f_\epsilon$ en la projetant, cellule par cellule, sur la fonction propre $\phi_0$. La composante orthogonale à $\phi_0$, nécessairement non nulle pour que $f_\epsilon$ ait un support compact, empêchera d'avoir $\tilde\lambda_0 = \lambda_0$. Nous notons toujours $\phi_0$ l'extension à $M$ de $\phi_0$, bien définie et continue par hypothèse. Pour chaque cellule $C_i$ de $M$, posons
	$$a_i^2 = \norm{f_\epsilon}^2_{C_i} = \int_{C_i}f_\epsilon^2,\ b_i = \langle f_\epsilon, \phi_0\rangle_{C_i} = \int_{C_i}f_\epsilon\phi_0,$$
	et
	$$c_i^2 = \norm{f_\epsilon-b_i\phi_0}^2_{C_i} = \int_{C_i}(f_\epsilon-b_i\phi_0)^2.$$
	On a alors $a_i^2 = b_i^2+c_i^2$. Nous noterons désormais $g_i = g_{i,\epsilon} = f_\epsilon-b_i\phi_0$ la composante de $f_\epsilon$ orthogonale à $\phi_0$ pour le produit scalaire $\Cl L^2$ sur $C_i$.
	
	Sur chaque cellule $C_i$, comme $\phi_0$ vérifie les conditions de Neumann,
	$$\int_{C_i} \nabla \phi_0.\nabla g_i = \int_{C_i}\Delta \phi_0 g_i = \lambda_0\int_{C_i} \phi_0 g_i = 0.$$
	On a donc
	$$\frac{\int_M |\nabla f_\epsilon|^2}{\int_M f_\epsilon^2} = \frac{\sum_i \norm{\nabla f_\epsilon}^2_{C_i}}{\sum_i \norm{f_\epsilon}^2_{C_i}} = \frac{\sum_i b_i^2\norm{\nabla \phi_0}^2_{C_i}+\norm{\nabla g_i}^2_{C_i}}{\sum_i a_i^2} = \frac{\sum_i \lambda_0b_i^2+\norm{\nabla g_i}^2_{C_i}}{\sum_i a_i^2}.$$
	Comme $g_i$ est orthogonale à $\phi_0$, $$\norm{\nabla g_i}^2_{C_i}\geq \lambda_1 \norm{g_i}^2_{C_i} = \lambda_1 c_i^2.$$ On obtient alors 
	$$ \norm{\nabla g_i}^2_{C_i} = \frac{\lambda_0}{\lambda_1}\norm{\nabla g_i}^2_{C_i}+\frac{\eta}{\lambda_1}\norm{\nabla g_i}^2_{C_i}\geq \lambda_0c_i^2+\frac{\eta}{\lambda_1}\norm{\nabla g_i}^2_{C_i},
	$$
	d'où
	$$\tilde \lambda_0+\epsilon\geq \frac{\int_M |\nabla f_\epsilon|^2}{\int_M f_\epsilon^2}\geq \frac{\sum_i \lambda_0b_i^2+\lambda_0c_i^2+\frac{\eta}{\lambda_1}\norm{\nabla g_i}^2_{C_i}}{\sum_i a_i^2} = \lambda_0 + \frac{\eta}{\lambda_1}\frac{\sum_i\norm{\nabla g_i}^2_{C_i}}{\sum_i a_i^2}.$$
	On a donc en particulier
	\beq \label{eq:MinNMoy} \delta +\epsilon = \tilde\lambda_0-\lambda_0 +\epsilon\geq \frac{\eta}{\lambda_1}\frac{\sum_i\norm{\nabla g_i}^2_{C_i}}{\sum_i a_i^2}.\eeq
	
	Nous allons donc nous intéresser au terme $\frac{\sum_i\norm{\nabla g_i}^2_{C_i}}{\sum_i a_i^2}$ et montrer le lemme :
	
	\begin{lemm}\label{lemm:AnaComb}
	Il existe une constante $A$ ne dépendant que de la géométrie et des valeurs de $\phi_0$ sur les zones de transition telle que
	$$\sum_i\norm{\nabla g_i}_{C_i}^2\geq \lambda_1A\sum_{i\sim j}(b_i-b_j)^2.$$
	\end{lemm}
	
	\begin{proof}
	
	Soient $i,j\in G$ tels que $i\sim j$ et $\alpha_{ij}$ la zone de transition commune aux deux cellules $C_i$ et $C_j$ (voir Définition \ref{def:GPeriod}). Par hypothèse, il existe dans $C_i\cup C_j$ un voisinage normal $T_{ij}$ de $\alpha_{ij}$ qui s'écrit en coordonnées de Fermi $T_{ij} = \alpha_{ij}\cx[-R,R].$ On  suppose que $(x,r)\in T_{ij}\cap C_i$ si et seulement si $r\geq 0$, et on note $T_{ij}^+$ la partie de $T_{ij}$ située dans $C_i$, et $T_{ij}^-$ l'autre. Pour obtenir une minoration, nous pouvons nous limiter à l'étude de ce qu'il se passe sur les $T_{ij}$. On a alors :
	
		$$  \sum_i\norm{\nabla g_i}^2_{C_i}\geq \frac{1}{2}\sum_i(\norm{\nabla g_i}^2_{C_i}+\lambda_1\norm{g_i}^2_{C_i})$$
		soit
\beq\label{eq:MinLoc}\sum_i\norm{\nabla g_i}^2_{C_i}\geq \frac{1}{2}\sum_{i\sim j}\left(\norm{\nabla g_i}^2_{T_{ij}^+}+\norm{\nabla g_j}^2_{T_{ij}^-}+\lambda_1(\norm{g_i}^2_{T_{ij}^+}+\norm{g_j}^2_{T_{ij}^-})\right).\eeq
	
	Nous poursuivons alors notre minoration par le lemme suivant :
	
	\begin{lemm}\label{lemm:AnaComb2}
	Il existe une constante $A$ ne dépendant que de la géométrie et des valeurs de $\phi_0$ sur les zones de transition telle que pour tous $i,j\in G$ tels que $i\sim j$,  
	$$\norm{\nabla g_i}^2_{T_{ij^+}}+\norm{\nabla g_j}^2_{T_{ij}^-}+\lambda_1(\norm{g_i}^2_{T_{ij}^+}+\norm{g_j}^2_{T_{ij^-}})\geq 2A\lambda_1(b_i-b_j)^2.$$
	\end{lemm}
	
	\begin{proof}
	Soient $i,j\in G$ tels que $i\sim j$. Comme à la section \ref{sec:PrelimGeom}, on note $(x,r)\in \alpha_{ij}\cx [-R,R]$ les coordonnées de Fermi sur $T_{ij}$, $dV_\alpha(x) = dV_{g_{\alpha}}(x)$ l'élément de volume sur $\alpha_{ij}$, et 
	$\theta^{n-1}(x,r)dV_\alpha(x)dr$ l'élément de volume sur $T_{ij}$. Soit $r\in[-R,R],$ posons
	$$V_{ij}(r) = \int_{\alpha_{ij}}\theta^{n-1}(x,r)dV_\alpha(x)$$
	l'aire de l'hypersurface parallèle à $\alpha_{ij}$ à la distance $r$. On a donc $V_{ij}(0) = \Vol(\alpha_{ij}).$

	Pour nous ramener à un problème ne dépendant que de $r$, pour chaque fonction $f$ définie sur $T_{ij}$, notons en lettre capitale sa \emph{moyenne horizontale} :
	$$F(r) = \frac{1}{V_{ij}(r)}\int_{\alpha_{ij}}f(x,r)\theta^{n-1}(x,r)dV_\alpha(x).$$ 
	
	Comme $\phi_0$ se recolle bien, $\phi_0(x,0)$ est bien définie. Par continuité de $f_\epsilon$, on a donc
	$$g_j(x,0)+b_j \phi_0(x,0) = g_i(x,0)+b_i \phi_0(x,0),$$
	d'où $$|g_j(x,0)-g_i(x,0)| = |b_j-b_i| \phi_0(x,0).$$
	Par intégration, on obtient alors
	$$|G_j(0)-G_i(0)| = |b_j-b_i| \Phi_0(0),$$
	 où l'on a donc noté
	$$G_i(r) = \frac{1}{V_{ij}(r)}\int_{\alpha_{ij}}g_i(x,r)\theta^{n-1}(x,r)dV_\alpha(x),$$  
	et $G_j$ et $\Phi$ sont définies de la même façon.
	
	On sait d'après le théorème \ref{theo:EigFunc} que $\Phi_0>0$. Supposons, quitte à inverser $i$ et $j$ et à changer le signe de $G_i$, que $$G_i(0)\geq \frac{1}{2}|b_j-b_i| \Phi_0(0).$$ 
Nous notons $$R_0 = \inf\left\{r\in ]0,R] : G_i(r)\leq \frac{1}{4}\abs{b_i-b_j}\Phi_0(0).\right\}.$$

\textbf{$1^{er}$ cas :  $R_0 = R$}\\
On a donc 
\beq\label{eq:MiniG}
\norm{G_i}_{T_{ij}^+} = \int_0^RG_i(r)^2V_{ij}(r)dr\geq \frac{\Phi_0(0)^2\Vol(T_{ij}^+)}{16}(b_i-b_j)^2.\eeq
Or, pour tout $r\in[0,R]$, d'après l'inégalité de Cauchy-Schwarz associée à la norme euclidienne $$f\fa\sqrt{\int_{\alpha_{ij}}f^2(x,r)\theta^{n-1}(x,r)dV_\alpha(x)},$$ on a
$$\left(\int_{\alpha_{ij}}g_i(x,r)\theta^{n-1}(x,r)dV_\alpha(x)\right)^2\leq \int_{\alpha_{ij}}1^2\theta^{n-1}(x,r)dV_\alpha(x)\int_{\alpha_{ij}}g_i(x,r)^2\theta^{n-1}(x,r)dV_\alpha(x),$$
soit
$$\left(\int_{\alpha_{ij}}g_i(x,r)\theta^{n-1}(x,r)dV_\alpha(x)\right)^2\leq V_{ij}(r)\int_{\alpha_{ij}}g_i(x,r)^2\theta^{n-1}(x,r)dV_\alpha(x).$$
On a donc
$$\norm{g_i^2}_{T_{ij}^+} = \int_0^R\int_{\alpha_{ij}}g_i(x,r)^2\theta^{n-1}(x,r)dV_\alpha(x)dr\geq \int_0^R\frac{\left(\int_{\alpha_{ij}}g_i(x,r)\theta^{n-1}(x,r)dV_\alpha(x)\right)^2}{V_{ij}(r)}dr$$
$$= \int_0^RG_i(r)^2V_{ij}(r)dr = \norm{G_i}_{T_{ij}^+} \geq A_1(b_i-b_j)^2$$
d'après l'inégalité (\ref{eq:MiniG}), avec
$$A_1 = \frac{\Phi_0(0)^2\Vol(T_{ij}^+)}{16}.$$
\pgh

\textbf{$2^{eme}$ cas :  $R_0<R$}
On a donc par définition de $R_0$, $$G_i(R_0)= \frac{1}{4}|b_j-b_i| \Phi_0(0) :$$
nous avons une minoration de la variation de la fonction $G_i$. Nous allons en déduire une minoration de la norme de son gradient par le lemme suivant, qui est le point technique principal de notre démonstration :

	\begin{lemm}\label{lemm:MinEnerg}
	Il existe une constante $A_2>0$ ne dépendant que de $\lambda_0$, de la géométrie et des valeurs de $\phi_0$ sur $T_{ij}$ telle que
	$$\int_{T_{ij}}|\nabla G_i|^2\geq A_2(b_i-b_j)^2.$$
	\end{lemm}
	
	\begin{proof}
	Posons 
	$$P= G_i(0) \geq \frac{1}{2}|b_j-b_i|\Phi_0(0)$$ et 
	$$Q = G_i(R_0) = \frac{1}{4}|b_j-b_i|\Phi_0(0).$$ 
	
	Nous noterons désormais
	$$T = T_{ij}^+\cap (\alpha_{ij}\cx[0,R_0]),$$
	$g_T$ la métrique induite de celle de $M$, et $dV_T$ son élément de volume.
	Soit $G_T(x,r)$ la fonction harmonique pour le Laplacien issu de la métrique de $M$ sur $T$, vérifiant $G_T(x,0) = P$ et $G_T(x,R_0) = Q$, et vérifiant les conditions de Neumann sur $\bd\alpha_{ij}\cx[0,R_0]$ lorsque $\bd\alpha_{ij}\neq \vd$. Parmi toutes les fonctions $f$ de $\Cl H^1(T)$ vérifiant $f(x,0) = P$ et $f(x,R_0) = Q$, $G_T$ minimise l'énergie de Dirichlet : on a donc
	$$\int_{T_{ij}^+}|\nabla G_i|²\geq \int_{T}|\nabla G_i|²dV_T\geq \int_{T}|\nabla G_T|²dV_T$$
	Il suffit donc de minorer l'énergie de $G_T$ pour démontrer le lemme \ref{lemm:MinEnerg}. Nous allons le faire par comparaison avec le cas où $T$ est muni d'une métrique que nous savons mieux contrôler. 
		
Notons $T_{\inf}$ le tube $T = \alpha_{ij}\cx [0,R_0]$ muni de la métrique (en coordonnées de Fermi)
	$$g_{\inf}(x,r) = (\theta_{\inf}(r))^2 g_{\alpha} \oplus dr^2,$$ 
	où $g_{\alpha}$ est la métrique sur $\alpha_{ij}$ induite par celle de $M$ et où $\theta_{\inf}$ vérifie les hypothèses du Corollaire \ref{coro:MinLap} : pour tout $x\in\alpha$, la fonction
	$$r\fa \frac{\theta_{\inf}(r)}{\theta(x,r)}$$ est décroissante. On a vu à la Remarque \ref{rem:ThetaInf} qu'une telle fonction existe toujours.
	
	\begin{prop}\label{prop:VolLapTransv}
	L'élément de volume sur $T_{\inf}$ s'écrit 
	$$dV_{\inf} = (\theta_{\inf}(r))^{n-1}dV_{\alpha}(x)dr.$$
	Soit $f$ une fonction sur $T_{\inf}$ ne dépendant que de la coordonnée $r$, son Laplacien s'écrit
	$$\Delta_{\inf} f = - f''(s)-(n-1)\frac{\theta'_{\inf}}{\theta_{\inf}}(r)f'(r).$$
	\end{prop}
Dans cette écriture, on note encore $f' = \frac{df}{dr}.$
	\begin{proof}
	Il suffit de recopier la preuve du corollaire \ref{coro:MinLap}. Les propositions de \cite{GHL04} utilisées restent valables pour la métrique $g_{\inf}$. En particulier, lorsqu'une fonction $f$ ne dépend que de $r$, $\Delta_{\inf}f$ est également radiale.
	\end{proof}
	
	Soit maintenant $G_{\inf}(s)$ l'unique fonction définie sur $T_{\inf}$, valant $P$ en $(x,0)$, $Q$ en $(x,R_0)$ pour tout $x\in \alpha_{ij}$, vérifiant les conditions de Neumann sur $\bd\alpha_{ij}\cx[0,R_0]$ si $\bd\alpha_{ij}\neq \vd$, et harmonique pour le Laplacien $\Delta_{\inf}$. On vérifie par unicité que $G_{\inf}$ ne dépend que de $s$ et s'écrit en coordonnées de Fermi :
	$$G_{\inf}(r) = P-\frac{P-Q}{U_{\inf}(R_0)}U_{\inf}(r)$$ avec $$U_{\inf}(r) = \int_0^r\frac{du}{\theta_{\inf}(u)^{n-1}}.$$ En particulier, comme $P>Q$, $G_{\inf}$ est décroissante en $r$.
	
	\begin{lemm}\label{lemm:CompFoncHarm}
	Avec les notations précédentes,
	$$\int_{T}|\nabla G_T|^2dV_T \geq \int_{T_{\inf}}|\nabla G_{\inf}|^2 dV_{\inf}.$$
	\end{lemm}
	\begin{proof}
	Par hypothèse, $G_T$ est harmonique : $\Delta_T G_T = 0.$ De plus, d'après le Corollaire \ref{coro:MinLap}, comme $G_{\inf}$ ne dépend que de $s$ et est décroissante, on a
	$$\Delta_T G_{\inf} = -\frac{d^2G_{\inf}}{dr^2} - (n-1)\frac{\theta'}{\theta}G_{\inf}'\geq \Delta_{\inf} G_{\inf} = 0.$$
	On a donc sur $T$ :
	$$\Delta_T(G_T-G_{\inf})\leq 0.$$ 
D'après le Principe du Maximum, $G_T-G_{\inf}$ est alors maximale pour $r=0$ ou $r=R_0$. Or, $$(G_T-G_{\inf})(x,0) = (G_T-G_{\inf})(x,R_0) = 0 :$$
	on a donc $\forall r\in[0,R_0],$
	$$(G_T-G_{\inf})(x,r)\leq 0.$$
De plus, par le principe du maximum fort (voir \cite{ProWei84} chapitre 2), on sait qu'alors
	$$\frac{\bd G_T}{\bd r}(x,0)<\frac{dG_{\inf}}{dr}(0) \mbox{	et } \frac{\bd G_T}{\bd r}(x,R_0)>\frac{dG_{\inf}}{dr}(R_0).$$
		On a d'après la Formule de Green
	\beq\label{eq:GreenTube}
	\int_T|\nabla G_T|^2dV_T = \int_TG_T(x,r)\Delta_T G_T(x,r) dV_T(x,r) - \int_{\alpha_{ij}}G_T (x,0)\frac{\bd G_T}{\bd r}(x,0)dV_{\alpha}(x) $$
	$$+ \int_{\alpha_{ij}}G_T(x,R_0) \frac{\bd G_T}{\bd r}(x,R_0)\theta^{n-1}(x,R_0)dV_{\alpha}(x).\eeq
En effet, comme $G_T$ et $G_{\inf}$ vérifient les conditions de Neumann sur $\bd\alpha_{ij}\cx[0,R_0]$, les termes de bord associés sont nuls. Le premier terme de cette équation est également nul, par définition de $G_T$. De plus, la fonction
	$$r\ra \int_{\alpha_{ij}}\frac{\bd G_T}{\bd r}(x,r)\theta^{n-1}(x,r)dV_{\alpha}(x)$$
est constante : il suffit d'appliquer la formule de Green sur le domaine $\alpha_{ij}\cx [0,r]$ avec $f = 1$ et $g = G_T$. On a donc

\beq\label{eq:GreenCourt}
\int_T|\nabla G_T|^2dV_T = \int_{\alpha_{ij}}\frac{\bd G_T}{\bd r}(0)dV_{\alpha}(x)(G_T(R_0) - G_T(0)) = \int_{\alpha_{ij}}\frac{\bd G_T}{\bd r}(0)dV_{\alpha}(x)(Q-P).\eeq
On a alors d'après le principe du maximum fort cité ci-dessus,
	$$\int_{\alpha_{ij}}\frac{\bd G_T}{\bd r}(x,0)dV_{\alpha}(x)(Q - P)>\int_{\alpha_{ij}}\frac{d G_{\inf}}{d r}(0)dV_{\alpha}(x)(Q - P)$$
	$$ = \int_{\alpha_{ij}}\frac{d G_{\inf}}{d r}(0)dV_{\alpha}(x)(G_{\inf}(R_0) - G_{\inf}(0)) = \int_{T_{\inf}}|\nabla G_{\inf}|^2dV_{\inf}$$
	en utilisant la formule de Green dans $T_{\inf}$ et l'équation (\ref{eq:GreenCourt}) adaptée au tube $T_{\inf}$. On a donc bien 
	$$ \int_T|\nabla G_T|^2dV_T \geq \int_{T_{\inf}}|\nabla G_{\inf}|^2dV_{\inf}.$$
\end{proof}
	
	Il nous suffit donc maintenant, pour prouver le lemme \ref{lemm:MinEnerg}, de minorer l'énergie de $G_{\inf}$ sur $T_{\inf}$.

	\begin{lemm}\label{lemm:MinEnergK}
	Avec les notations précédentes, il existe une constante $A_2>0$ telle que
	$$\int_{T_{\inf}}|\nabla G_{\inf}|^2dV_{\inf} \geq A_2(b_i-b_j)^2.$$
	\end{lemm}

	\begin{proof}
	Nous connaissons une écriture (presque) explicite de $G_{\inf}$ :
	$$G_{\inf}(r) = P-\frac{P-Q}{U_{\inf}(R_0)}U_{\inf}(r)$$ avec $$U_{\inf}(r) = \int_0^r\frac{du}{\theta_{\inf}^{n-1}(u)}.$$
	ainsi que l'élément de volume 
	$$dV_{\inf}(x,r) = \theta_{\inf}^{n-1}(r)dV_\alpha(x)dr.$$
	On a donc
	$$\int_{T_{\inf}}|\nabla G_{\inf}|^2dV_{\inf} = \int_0^{R_0}\int_{\alpha_ij}\left(G_{\inf}'(r)\right)^2\theta_{\inf}^{n-1}(r)dV_\alpha(x)dr $$
	$$= \Vol(\alpha_{ij})\frac{(P-Q)^2}{U_{\inf}(R_0)^2}\int_0^{R_0}\left(U_{\inf}'(r)\right)^2\theta_{\inf}^{n-1}(r)dr.$$
	Or, $U_{\inf}'(r) = \frac{1}{\theta_{\inf}^{n-1}(r)},$ donc 
	$$\int_0^{R_0}\left(U_{\inf}'(r)\right)^2\theta_{\inf}^{n-1}(r)dr = \int_0^{R_0} \frac{1}{\theta_{\inf}^{n-1}(r)} dr = U(R_0)>0.$$
	On a alors
	$$ \frac{1}{U_{\inf}(R_0)^2}\int_0^{R_0} (U'_{\inf}(r))^2\theta_{\inf}^{n-1}(r)dr = \frac{1}{U_{\inf}(R_0)}>\frac{1}{U_{\inf}(R)}>0$$ où $R$ est la largeur du voisinage tubulaire de la zone de transition, car $U_{\inf}$ est évidemment croissante.
On a donc finalement
	$$\int_{T_{\inf}}|\nabla G_{\inf}|^2dV_{\inf} = \Vol(\alpha_{ij})\frac{(P-Q)^2}{U_{\inf}(R_0)}\geq \frac{\Vol(\alpha_{ij})\Phi_0^2(0)}{16U_{\inf}(R)}(b_i-b_j)^2,$$
	car $$(P-Q)^2\geq\frac{\Phi_0^2(0)(b_i-b_j)^2}{4}.$$
	
Ceci conclut la démonstration de notre lemme en posant 
	$$A_2 = \frac{\Vol(\alpha_{ij})\Phi_0^2(0)}{16U_{\inf}(R)} = \frac{\Vol(\alpha_{ij})\Phi_0^2(0)}{16\int_0^R\frac{du}{\theta_{\inf}^{n-1}(u)}},$$
qui ne dépend que de la géométrie du tube $T$ et des valeurs de $\phi_0$ sur $\alpha$. On peut écrire cette constante plus explicitement dès que l'on connait une expression de $\theta$ ; par exemple lorsque la courbure sectionnelle de $M$ et la courbure moyenne de $\alpha$ sont constantes.
	\end{proof}

	D'après le lemme \ref{lemm:CompFoncHarm}, cette inégalité implique 
	$$\int_{T}|\nabla G_T|^2dV_T \geq A_2(b_i-b_j)^2,$$
	qui elle même implique le lemme \ref{lemm:MinEnerg}.
	\end{proof}
	
Pour conclure la démonstration du Lemme \ref{lemm:AnaComb2}, il nous faut minorer l'énergie de $g_i$. Nous aurons donc besoin d'une expression plus précise de son gradient qui la relie à l'énergie de $G_i$ que nous venons de minorer. Comme par construction des coordonnées de Fermi $\frac{\bd}{\bd r}$ est normalisé, on a
	$$ \int_{T_{ij}^+}|\nabla g_i|^2 \geq \int_0^{R}\int_{\alpha_{ij}}\left(\frac{\bd g_i}{\bd r}\right)^2\theta^{n-1}(x,r)dV_\alpha(x)dr.$$
Nous obtenons alors, de nouveau à l'aide de l'inégalité de Cauchy-Schwarz :
	$$\int_{T_{ij}^+}|\nabla g_i|^2\geq \int_{0}^{R}\frac{\left(\int_{\alpha_{ij}}\frac{\bd g_i}{\bd r}\theta^{n-1}(x,r)dV_\alpha(x)\right)^2}{V_{ij}(r)} dr.$$
De plus, on a
	$$\frac{d G_i}{d r}(r) = \frac{\bd }{\bd r}\left(\frac{1}{V_{ij}(r)}\int_{\alpha_{ij}} g_i(x,r)\theta^{n-1}(x,r)dV_\alpha(x)\right) = \frac{1}{V_{ij}(r)}\int_{\alpha_{ij}} \frac{\bd g_i(x,r)}{\bd r}\theta^{n-1}(x,r)dV_\alpha(x)$$ $$+ \frac{1}{V_{ij}(r)}\int_{\alpha_{ij}}g_i(x,r)\left[\frac{1}{\theta^{n-1}(x,r)}\frac{\bd (\theta^{n-1}(x,r))}{\bd r}-\frac{V_{ij}'(r)}{V_{ij}(r)}\right]\theta^{n-1}(x,r) dV_\alpha(x)$$
	$$ = \frac{1}{V_{ij}(r)}\int_{\alpha_{ij}} \frac{\bd g_i(x,r)}{\bd r}\theta^{n-1}(x,r)dV_\alpha(x) + \frac{1}{V_{ij}(r)}\int_{\alpha_{ij}}g_i(x,r)\beta_{ij}(x,r)\theta^{n-1}(x,r)dV_\alpha(x),$$
	où $$\beta_{ij}(x,r) = \frac{1}{\theta^{n-1}(x,r)}\frac{\bd (\theta^{n-1}(x,r))}{\bd r}-\frac{V_{ij}'(r)}{V_{ij}(r)}$$ est la \emph{fonction d'oscillation} de l'élément de volume $\theta$ que nous avons définie au paragraphe \ref{sec:PrelimGeom}. D'après ci-dessus, on a donc
	\beq\label{eq:MinGi}
	\int_{T_{ij}^+}|\nabla g_i|^2\geq \int_{0}^{R}\left(\frac{d G_i}{d r} - \frac{1}{V_{ij}(r)}\int_{\alpha_{ij}}g_i(x,r)\beta_{ij}(x,r)\theta^{n-1}(x,r)dV_\alpha(x)\right)^2V_{ij}(r)dr.\eeq
	
Or, comme pour tout $r\in[0,R]$, l'aire de l'hypersurface à distance $r$ de $\alpha_{ij}$ est strictement positif :
$V_{ij}(r)>0,$ l'application
	$$f\fa \sqrt{\int_0^R f^2(r) V_{ij}(r)dr}$$
définit une norme euclidienne sur l'ensemble des fonctions continues de $[0,R]$ dans $\Bb R$. D'après l'inégalité triangulaire, on a alors

$$\sqrt{\int_{T_{ij}^+}|\nabla g_i|^2}\geq \sqrt{\int_{0}^{R}\left|\frac{d G_i}{d r}\right|^2V_{ij}(r)dr} $$
\beq\label{eq:MinGi2}	- \sqrt{\int_0^R\left|\frac{1}{V_{ij}(r)}\int_{\alpha_{ij}}g_i(x,r)\beta_{ij}(x,r)\theta^{n-1}(x,r)dV_\alpha(x)\right|^2V_{ij}(r)dr}.\eeq	

\pgh
\textbf{Cas 2.a :
$$\sqrt{\int_{0}^{R}\left|\frac{d G_i}{d r}\right|^2V_{ij}(r)dr} - \sqrt{\int_0^R\left|\frac{1}{V_{ij}(r)}\int_{\alpha_{ij}}g_i(x,r)\beta_{ij}(x,r)\theta^{n-1}(x,r)dV_\alpha(x)\right|^2V_{ij}(r)dr}$$
$$\geq \frac{1}{2}\sqrt{\int_{0}^{R}\left|\frac{d G_i}{d r}\right|^2V_{ij}(r)dr}$$	}
L'inégalité (\ref{eq:MinGi2}) devient alors
$$\int_{T_{ij}^+}|\nabla g_i|^2\geq \frac{1}{4}\int_{T_{ij}}|\nabla G_i|^2\geq \frac{A_2}{4}(b_i-b_j)^2$$ d'après le Lemme \ref{lemm:MinEnerg}.

\pgh
\textbf{Cas 2.b :
$$\sqrt{\int_{0}^{R}\left|\frac{d G_i}{d r}\right|^2V_{ij}(r)dr} - \sqrt{\int_0^R\left|\frac{1}{V_{ij}(r)}\int_{\alpha_{ij}}g_i(x,r)\beta_{ij}(x,r)\theta^{n-1}(x,r)dV_\alpha(x)\right|^2V_{ij}(r)dr}$$
$$\leq \frac{1}{2}\sqrt{\int_{0}^{R}\left|\frac{d G_i}{d r}\right|^2V_{ij}(r)dr}$$	}	
On a alors
$$\sqrt{\int_0^R\left|\frac{1}{V_{ij}(r)}\int_{\alpha_{ij}}g_i(x,r)\beta_{ij}(x,r)\theta^{n-1}(x,r)dV_\alpha(x)\right|^2V_{ij}(r)dr}\geq \frac{1}{2}\sqrt{\int_{0}^{R}\left|\frac{d G_i}{d r}\right|^2V_{ij}(r)dr},$$	
d'où d'après le Lemme \ref{lemm:MinEnerg},
$$\int_0^R\frac{1}{V_{ij}(r)}\left|\int_{\alpha_{ij}}g_i(x,r)\beta_{ij}(x,r)\theta^{n-1}(x,r)dV_\alpha(x)\right|^2V_{ij}(r)dr\geq \frac{A_2(b_i-b_j)^2}{4}.$$
En utilisant de nouveau l'inégalité de Cauchy-Schwarz sur chaque tranche horizontale, cette majoration devient
$$\frac{A_2(b_i-b_j)^2}{4}\leq \int_0^R\left( \int_{\alpha_{ij}}g_i(x,r)^2\theta^{n-1}(x,r)dV_\alpha(x)\frac{\int_{\alpha_{ij}}\beta_{ij}(x,r)^2\theta^{n-1}(x,r)dV_\alpha(x)}{V_{ij}(r)}\right)dr$$
$$\leq \Kappa^2\int_0^R \int_{\alpha_{ij}}g_i(x,r)^2\theta^{n-1}(x,r)dV_\alpha(x)dr,$$
où 
$$\Kappa = \sup_{(x,r)\in T_{ij}^+}\abs{\beta_{ij}(x,r)}<+\infty$$
est finie car $T_{ij}$ est compact, et $\Kappa$ ne dépend que de la géométrie de $T_{ij}$. On a donc
$$\norm{g_i}_{T_{ij}^+}^2\geq A_3(b_i-b_j)^2,$$
avec
$$A_3 = \frac{A_2}{4\Kappa^2} = \frac{\Phi_0(0)^2\Vol(\alpha_{ij}\cx[0,R])}{64\Kappa^2}.$$
 En posant $$2A = \min\left\{A_1, \frac{A_2}{4\lambda_1}, A_3\right\},$$ qui ne dépend que de la géométrie du tube $T_{ij}$ et des valeurs de $\phi_0$ sur $\alpha_{ij}$,	nous obtenons donc
	$$\norm{\nabla g_i}^2_{T_{ij^+}}+\norm{\nabla g_j}^2_{T_{ij}^-}+\lambda_1(\norm{g_i}^2_{T_{ij}^+}+\norm{g_j}^2_{T_{ij^-}})\geq 2\lambda_1A(b_i-b_j)^2.$$
	Ceci conclut la preuve du lemme \ref{lemm:AnaComb2}
	
\begin{rema}
Lorsque la métrique au voisinage de $\alpha$ ne dépend que de la coordonnée radiale $r$, par exemple si $\alpha$ est une hypersurface totalement géodésique dans une variété à courbure sectionnelle constante, la fonction d'oscillation $\beta$ est nulle, le cas 2.b n'existe plus et toute notre démonstration se trouve largement simplifiée. On peut alors écrire explicitement toutes nos constantes en fonction de l'expression de la métrique en coordonnées de Fermi et de la moyenne de $\Phi_0$ sur les zones de transition.
\end{rema}
 
	\end{proof}

La minoration (\ref{eq:MinLoc}) devient alors
	$$\sum_i\norm{\nabla g_i}^2\geq \lambda_1 A\sum_{i\sim j}(b_i-b_j)^2,$$
	où $A$ défini ci-dessus ne dépend que de $\Kappa$, $R$, $\Vol(T_{ij}), \Vol(\alpha_{ij})$, $\lambda_1$ et $\Phi_0(0)$, ce qui conclut la démonstration du Lemme \ref{lemm:AnaComb}.

	\end{proof}

	Achevons la preuve du Théorème \ref{theo:ControleGene}. D'après l'inégalité (\ref{eq:MinNMoy}),
	$$\frac{\eta}{\lambda_1}\sum_i\norm{\nabla g_i}^2_{C_i}\leq (\delta+\epsilon)\norm{f_\epsilon}^2 = (\delta+\epsilon)\sum_i a_i^2,$$
	donc
	$$\eta\sum_i\norm{g_i}^2_{C_i} = \eta\sum_i c_i^2\leq (\delta+\epsilon)\sum_i a_i^2.$$
	On a alors
	$$\sum_i a_i^2 = \sum_i b_i^2+\sum_i c_i^2\leq \sum_i b_i^2+\left(\frac{\delta+\epsilon}{\eta}\right)\sum_i a_i^2.$$

Si $\delta\geq \eta$, comme $\mu_0(G)\leq v$ (cela découle directement de la définition, voir paragraphe \ref{sec:PrelimMoy}), on a
$$\delta\geq \frac{\eta}{v}\mu_0(G),$$ ce qui conclut évidemment notre démonstration (indépendamment des lemmes précédents). Sinon, pour $\epsilon>0$ suffisamment petit, on a
$$\frac{\delta+\epsilon}{\eta}<1.$$
On obtient alors $$\sum_i a_i^2\leq \frac{1}{1-\frac{\delta+\epsilon}{\eta}}\sum_i b_i^2.$$
	L'inégalité (\ref{eq:MinNMoy}) devient
	$$\delta + \epsilon\geq\frac{\eta}{\lambda_1}\frac{\sum_i\norm{\nabla g_i}^2_{C_i}}{\sum_i a_i^2}\geq (1-\frac{\delta+\epsilon}{\eta})\frac{\eta \lambda_1 A}{\lambda_1}\frac{\sum_{i\sim j}(b_i-b_j)^2}{\sum_i b_i^2}.$$
	Nous rappelons que par définition, $$\mu_0 = \inf \frac{\sum_{i\sim j}(\alpha_i-\alpha_j)^2}{\sum_i \alpha_i^2},$$ où $(\alpha_i)$ parcourt l'ensemble des familles positives à support compact dans $G$ (voir section \ref{sec:PrelimMoy}). On obtient donc $\delta\geq (\eta-\delta) A\mu_0,$ lorsque $\epsilon$ tend vers $0$, soit 
	$$\delta(1+A\mu_0)\geq A\eta\mu_0,$$
	soit finalement
	\beq \label{eq:MinDelta} \delta\geq \frac{\eta A}{1+A\mu_0}\mu_0\geq \eta A \mu_0,\eeq
	où $A$ ne dépend que de $\lambda_0$, de $\Phi_0(0)$ et de la géométrie des zones de transition d'après les lemmes précédents.
	\end{proof}

\section{Applications et Généralisations}

\subsection{Bas du spectre des revêtements}\label{ssec:Brooks}

Nous avons développé cette théorie des variétés modelées sur un graphe entre autres pour généraliser et préciser certains résultats obtenus par R. Brooks, que nous présentons maintenant. Dans l'article \cite{Broo85Reine}, il démontre le résultat suivant :

\begin{theo}[Brooks, 1984]
Soit $p : M\ra N$ un revêtement galoisien, et $\phi_0$ une fonction $\lambda_0(N)$-harmonique positive sur $M$. Soit $D\subset M$ un domaine fondamental pour $p$, supposons qu'il existe un compact $K\subset D$ tel que
$$\g h_{\phi_0}(D\bs K)>0.$$
Alors $\lambda_0(M) = \lambda_0(N)$ si et seulement si le groupe d'automorphisme de $p$ est moyennable.
\end{theo}

Pour un domaine $O\subset M$, la constante $\g h_{\phi_0}(O)$ ci-dessus est définie par
$$\g h_{\phi_0}(O) = \inf_S\left\{\frac{\int_S (\phi_0)^2dV_{n-1}}{\int_{int(S)} (\phi_0)^2dV_n}\right\},$$
où $S$ parcourt l'ensemble des hypersurfaces compactes strictement incluses dans $O$, et $int(S)$ désigne l'ouvert bordé par $S$ qui n'intersecte pas $\bd O$. Il s'agit d'une constante de Cheeger pour un Laplacien renormalisé ; R. Brooks montre que $\lambda_0(M) = \lambda_0(N)$ si et seulement si $\g h_{\phi_0}(M) = 0$. L'hypothèse $$\g h_{\phi_0}(D\bs K)>0, $$ est difficile à expliciter. Remarquons qu'il s'agit encore d'une hypothèse sur la géométrie d'un domaine fondamental. Elle implique, comme le souligne R. Brooks lui-même, que le trou spectral de $N$ est strictement positif. Le seul exemple qu'il traite entièrement R. Brooks est celui des variétés hyperboliques sans cusps. Une étude plus poussée de la démonstration de R. Brooks montre qu'elle n'est pas valable lorsque $\bd D$ n'est pas compact et son rayon d'injectivité normale tend vers $0$ dans une direction de $\bd D$. En particulier, dans le cas où $N$ est une variété hyperbolique et où le groupe d'automorphisme $\Gamma$ contient des éléments paraboliques (chaque domaine rencontre alors une infinité de ses copies au voisinage du cusp), sa démonstration n'est pas valable. Nous avons vu au paragraphe \ref{ssec:Result} que, dans le cas des revêtements, notre Théorème \ref{theo:ControleGene} implique le corollaire :

\begin{coro}
Soit $p : M\ra N$ un revêtement riemannien galoisien, dont le groupe d'automorphisme $\Gamma$ est de type fini. Supposons que $\lambda_0(N)<\lambda_1(N)$, et notons $\phi_0$ la première fonction propre associée à $\lambda_0(N)$. S'il existe un domaine fondamental $C$ pour l'action de $\Gamma$ tel que le relevé de $\phi_0$ vérifie les conditions de Neumann sur $\bd C$, alors 
$$\lambda_0(M)\geq \lambda_0(N),$$ avec égalité si et seulement si $\Gamma$ est moyennable. 
\end{coro}

Ainsi que nous l'avons déjà dit, les résultats de \cite{Tap09b} montrent qu'un tel domaine existe dès que les $\frac{n}{2}+2$ premières dérivées du tenseur de courbure sont uniformément bornées, et que la métrique est générique (pour une topologie bien choisie). Nous retrouvons donc un résultat identique à celui de R. Brooks lorsque la métrique est générique, avec des hypothèses géométriques plus habituelles qu'une borne sur la constante $\g h_\phi$. De plus, la version complète du Théorème \ref{theo:ControleGene} nous donne de surcroît un contrôle intéressant sur l'écart entre les bas du spectre à l'aide de la combinatoire du groupe d'automorphisme du revêtement, qui affine le résultat de R. Brooks. Enfin, notre construction s'applique aussi à d'autres cas que les revêtements : les cellules peuvent être quelconques, et les graphes ne sont pas nécessairement les graphes de Cayley d'un groupe. Lorsque la variété comprend des cusps, s'ils sont à l'intérieur d'un domaine fondamental et si l'on sait que la fonction propre est de Morse, notre méthode s'adapte assez bien (on montre que le gradient de la fonction propre reste simplement stratifié ; mais les notations deviennent très lourdes). En revanche, notre théorème de généricité nécessiterait des hypothèses plus précises sur les bouts à l'infini pour traiter le cas des variétés à cusps. Enfin, de même que dans les travaux de R. Brooks, nos résultats ne nous permettent pas pour le moment de traiter le cas où le groupe d'automorphisme du revêtement contient un élément parabolique.

\pgh
 
Pour obtenir un résultat complet, il suffirait que pour tout revêtement, il existe un domaine fondamental sur lequel la première fonction propre de la variété quotient vérifie les conditions de Neumann. Si nous pouvons montrer que la première fonction propre a un gradient simplement stratifié (voir Définition \ref{def:GradStrat}) sous quelques hypothèses géométriques (courbure minorée, rayon d'injectivité positif), nous obtiendrions donc une généralisation complète du résultat de R. Brooks. Si le gradient d'une fonction analytique est toujours stratifié (voir Section \ref{sec:CoinsRevet}, Question 1), notre méthode permet de retrouver l'exemple de R. Brooks dans le cas des variétés hyperboliques sans cusp.  

Il est raisonnable de penser qu'il n'est pas nécessaire d'avoir une métrique générique. On pourrait, par un argument d'approximation (voir par exemple la démonstration du Théorème 2 de \cite{Tap09b}), espérer utiliser le cas générique pour avoir le même résultat pour une métrique quelconque. Cependant, on a vu que la constante $A$ construite lors de la démonstration de la minoration de  $\lambda_0(M)-\lambda_0(N)$ dépend de façon cruciale de la géométrie des zones de transition. Dans le cas où $\phi_0$ n'est pas de Morse, lorsque l'on approxime la métrique par une suite de métriques génériques, les zones de transition correspondant aux fonctions propres de Morse ainsi obtenues ont toutes les chances de dégénérer. Il nous manque donc plus d'information sur le comportement du gradient de $\phi_0$ dans le cas général pour pouvoir trouver un domaine fondamental $C$ adapté. 

Il nous manque aussi une méthode permettant de traiter le cas où le bord du domaine a un rayon d'injectivité qui tend vers 0 à l'infini (par exemple lorsque le groupe d'automorphisme d'un revêtement entre deux variétés hyperboliques contient un élément parabolique). Cela nécessite une étude supplémentaire du comportement de la fonction propre dans ces <<bouts fins>> qui nous permettrait peut-être de trouver un domaine fondamental adapté. 

\subsection{Variétés à découpages bornés}\label{ssec:DecBorn}

Nous avons démontré notre théorème dans le cas où le graphe est à valence constante, et où toutes les cellules sont identiques. Lorsque les cellules sont de volume fini, la première fonction propre est constante (elle se recolle alors évidemment très bien !). Nous pouvons alors généraliser nos méthodes au cadre suivant :

\begin{defi}
Soit $M^n = \cup_{i\in I} M_i$ une partitions en morceaux d'intérieur non vide, de volume fini, uniformément majoré par $V_M$ et minoré par $V_m$, à bords $\Cl C^1$ par morceaux, d'intérieurs disjoints, tels que le graphe d'adjacence des $M_i$ (deux éléments sont reliés s'ils sont voisins au sens du paragraphe \ref{sec:CoinsRevet}) soit localement fini à valence bornée par $v>0$. On dit que le découpage de $M$ dans les $M_i$ est \emph{spectralement borné} s'il existe des constantes $\eta,V,R,\Kappa$ telles que pour tous $i,j\in I$, on ait :
\begin{enumerate}
\item le volume des $\bd M_i$ est inférieur à $V$ ;
\item le trou spectral de $M_i$ est supérieur à $\eta>0$ ;
\item pour tous $i,j\in G, i\sim j$, il existe des \ind{zones de transition}
$$\alpha_{ij}\subset \bd M_i\cap \bd M_j$$
ouvertes d'adhérences compactes, qui admettent des voisinages tubulaires $T_{ij}$ disjoints de largeur $R$ ;
\item il existe une fonction $\theta_{\inf} : [0,R]\ra \Bb R^*_+$, valant $1$ en $0$, telle que sur tout tube $T_{ij}$ au voisinage de la zone de transition $\alpha_{ij}$, l'élément de volume $\theta_{ij}$ vérifie, en coordonnées de Fermi :
$$r\fa \frac{\theta_{\inf}(r)}{\theta(x,r)}$$
est décroissante pour tout $x\in\alpha_{ij}$.
\item les fonctions d'oscillation de l'élément de volume (voir Définition \ref{def:VarVol}) sont uniformément bornées par $\Kappa$.
\end{enumerate}
\end{defi}
	
	Les méthodes développées à la section \ref{sec:Demo} permettent alors d'obtenir le théorème suivant :
	
	\begin{theo}
	Sous les hypothèses ci-dessus, il existe $A_1$ et $A_2$ ne dépendant que de $n, v, V_m, V,R,\Kappa$ et $\theta_{\inf}$ telles que 
	$$A_1\eta\mu_0(G)\leq \lambda_0(M)\leq A_2\g h(G).$$
	\end{theo}

\begin{proof}
	La preuve de la majoration est identique à celle du théorème \ref{theo:MajoGene} : il suffit pour cela de la reproduire en remplaçant $\phi_\epsilon$ par la fonction constante $1$ (qui se recolle évidemment). 

	La preuve de la minoration est légèrement différente : nous prenons ici $\phi_0 = 1$, qui n'est plus normalisée. Soit $f_\epsilon$ une fonction à support compact dans $M$ d'énergie $\epsilon-$proche de $\lambda_0 = \lambda_0(M)$. La discrétisation de $f_\epsilon$ sera alors définie, pour tout $i\in G$, par :

	$$a_i^2 = \frac{1}{\Vol(C_i)}\int_{C_i}f_\epsilon^2,\ b_i = \frac{1}{\Vol(C_i)}\int_{C_i}f_\epsilon,$$
	et
	$$c_i^2 = \frac{1}{\Vol(C_i)}\int_{C_i}(f_\epsilon-b_i)^2.$$
	On pose de nouveau $$g_i = f_\epsilon - b_i.$$
	Notre nouvelle normalisation permet que, de nouveau, $$a_i² = b_i² + c_i²,$$
	et le quotient de Rayleigh de $f_\epsilon$ s'écrit :
	\beq\label{eq:Energborn}
	\lambda_0(M)+\epsilon\geq \frac{\int_M |\nabla f_\epsilon|^2}{\int_M |f_\epsilon|^2} =\frac{\sum_i\norm{\nabla g_i}^2_{C_i}}{\sum_i a_i^2\Vol(C_i)}.\eeq

	L'étape clé de notre démonstration sera de nouveau le lemme suivant :
	\begin{lemm}
	 Il existe une constante $b$ ne dépendant que de $n, R, V_1, V_2$ et $\kappa$  telle que
	$$\sum_i\norm{\nabla g_i}^2_{C_i}\geq b\sum_{i\sim j} (b_i-b_j)².$$
	\end{lemm}

Le recollement de la fonction $1$ de par et d'autre de chaque composante de bord $\alpha_{ij}$ permet de reproduire exactement la preuve du lemme \ref{lemm:AnaComb}. Le reste de la démonstration s'adapte aisément.
\end{proof}

	On rappelle que d'après le théorème \ref{theo:CheegComb}, on a $\mu_0(G)\geq \frac{\g h(G)^2}{2v}.$ Ainsi, le résultat de ce théorème s'écrit encore :
	$$A'_1\g h(G)^2\leq \lambda_0(M)\leq C_2 \g h(G),$$
où le graphe $G$ correspond à celui d'un découpage de $M$ en morceaux de volume borné. Nous avons là un analogue discret de résultats plus connus de J. Cheeger et P. Buser : ils montrent par des méthodes analytiques que pour toute variété de dimension $n$ dont la courbure de Ricci est minorée par $-\kappa$,
	$$\frac{\g h²(M)}{4}\leq \lambda_0(M)\leq C(n,\kappa)\g h(M),$$
	avec $$\g h(M) = \inf\frac{\Vol_{n-1}(\bd M')}{\Vol_n(M')},$$
	où $M'$ décrit l'ensemble des hypersurfaces bordant un domaine compact de $M$. La minoration vient du classique \cite{Cheeg70} et est valable quelle que soit $M$. La majoration est initialement un résultat de P. Buser (\cite{Bus82}), mais sous cette formulation, on en trouvera une démonstration plus directe dans \cite{Cana92}. Cependant, pour évaluer effectivement $\lambda_0$ à partir de ces estimées de la constante de Cheeger, il est toujours délicat de trouver une famille d'hypersurfaces qui approchent $\g h(M)$ lorsque $M$ n'est pas compacte. Notre résultat permet d'obtenir assez simplement des bornes sur le bas du spectre, en particulier de montrer qu'il est strictement positif lorsqu'il existe un découpage borné non moyennable. Il suggère aussi une méthode de calcul approché de $\lambda_0(M)$ par discrétisation en éléments finis, où il suffit de contrôler le volume des éléments et de leurs bords. Lorsque la variété comporte des cusps, cette méthode est alors plus efficace que les méthodes traditionnelles de discrétisation avec des boules de rayon constant (comme par exemple dans \cite{Man05}), qui nécessitent de contrôler le comportement des fonctions propres à l'infini dans les cusps. En effet, dès que le <<morceau>> est de volume fini, notre méthode ne s'intéresse qu'à son bord relatif à l'intérieur de la variété.

\subsection{Perspectives}

On a déjà vu que pour pouvoir appliquer nos méthodes à tous les revêtements riemanniens, nous aurions besoin de construire en général un domaine fondamental sur lequel la première fonction propre de la base se relève en une fonction qui vérifie les conditions de Neumann, et qu'il serait suffisant pour cela de montrer que le gradient de la première fonction propre est simplement stratifié. Cette dernière question est intéressante en elle-même.
\pgh
La majoration que nous avons obtenue dans le cas où le graphe n'est pas moyennable et où chaque cellule n'admet qu'un nombre de voisins n'est certainement pas optimale. On peut espérer l'améliorer, en faisant intervenir des fonctions discrètes sur le graphe qui tendent vers un minimum pour l'énergie. On peut alors espérer obtenir, pour une variété $G$-périodique modelée sur une cellule $C$ n'ayant qu'un nombre fini de voisins,
$$A\eta\mu_0(G)\leq \lambda_0(M)-\lambda_0(C)\leq B\mu_0(G).$$
Ce serait particulièrement intéressant pour la méthode de calcul approchée du bas du spectre d'une variété admettant un découpage spectralement borné que nous venons de présenter. On se trouve alors exactement dans une généralisation des méthodes de \cite{CoulSal95}, où l'on discrétise une variété non compacte à l'aide de boules pondérées par leur volume.
\pgh
Notons qu'il n'est pas toujours facile d'expliciter les constantes $A$ et $A'$ qui figurent dans nos théorèmes principaux ; un cas cependant permet de les écrire de façon intéressante : lorsque $C$ est une surface de Riemann dont le bord est constitué de $v$ géodésiques fermées disjointes. Dans \cite{Colb85} et \cite{ColbCol88}, B. Colbois et Y. Colin de Verdière montrent que si l'on a un nombre fini de surfaces de Riemann à bords géodésiques reliés en suivant la combinatoire d'un graphe $G$, lorsque le graphe $G$ est fini, si l'on pince uniformément les géodésiques de jonction, le bas du spectre tend vers 0 et sa vitesse est précisément donnée par $\mu_0(G)$ et le volume des cellules. Notons que la méthode utilisée dans \cite{ColbCol88} (reposant sur le \emph{Lemme des petites valeurs propres}) ne s'applique pas aux graphes infinis. On peut espérer, à partir des méthodes que nous venons de présenter, généraliser les résultats de \cite{ColbCol88} aux variétés $G$-périodiques où $G$ est infini et la cellule est une surface de Riemann à bords géodésiques.
\pgh
Il serait aussi intéressant d'étudier, sur ces variétés périodiques généralisées, le spectre d'autres opérateurs, en particulier l'opérateur de Schrödinger : on se trouve alors dans la situation, physiquement intéressante, d'une macromolécule périodique dont on cherche les niveaux d'énergie.

\bibliographystyle{alpha}
\def\cprime{$'$}

\end{document}